\documentclass[a4paper,12pt]
{amsart}

\usepackage{amsmath}
\usepackage{amssymb,amsbsy,amsmath,amsfonts,amssymb,amscd}
\usepackage{enumitem}
\usepackage{latexsym}
\usepackage{amsthm}
\usepackage{graphics}
\usepackage{color}
\usepackage{tikz}
\usepackage{tikz-3dplot}
\usetikzlibrary{arrows}
\usepackage{tikz-cd}
\usepackage{pdfpages} 
\usepackage{longtable}
\usepackage{xy}
\usepackage{array}
\usepackage{color}
\usepackage{comment}
\usepackage{hyperref}

\input xy
\xyoption{all}

\usepackage{listings} 
 \lstdefinelanguage{Magma}%
  {%
   otherkeywords={:=,+:=,-:=,*:=},%
   procnamekeys={function,func,intrinsic,procedure,proc,return},%
   morekeywords={true,false},%
   morekeywords=[2]{adj,and,cat,cmpeq,cmpne,diff,div,eq,ge,gt,in,is,join,le,lt,%
          meet,mod,ne,notadj,notin,notsubset,or,sdiff,subset,xor},%
   morekeywords=[3]{assigned,break,by,case,catch,continue,declare,default,%
          delete,do,elif,else,end,eval,exists,exit,for,forall,fprintf,if,local,%
          not,print,printf,quit,random,read,readi,repeat,restore,save,select,%
          then,time,to,try,until,vprint,vprintf,vtime,when,where,while},%
   morekeywords=[4]{clear,forward,freeze,iload,import,load},%
   morekeywords=[5]{assert,assert2,assert3,error,require,requirege,requirerange},%
   morekeywords=[6]{car,comp,cop,elt,ext,frac,hom,ideal,iso,lideal,loc,map,%
          ncl,pmap,quo,rec,recformat,rep,rideal,sub},%
      sensitive,%
      morecomment=[l]//,%
      morecomment=[s]{/*}{*/},%
      morestring=[b]"%
  }[keywords,procnames,comments,strings]%
 
\lstnewenvironment{code_magma}[1][]
{\lstset{basicstyle=\scriptsize\ttfamily, columns=fullflexible,
language=Magma,
keywordstyle=\color{red}\bfseries,
commentstyle=\color{blue},tabsize=6,
numbers=left, numberstyle=\tiny,
stepnumber=1, numbersep=5pt, frame=single, #1}}{} 

\newcommand{\ks}{\mathsf{k}}

\def\red{\color[rgb]{1,0,0}}
\def\verde{\color[rgb]{.2,.5,.1}}

\newcommand\sC{{\mathcal C}}
\newcommand\sT{{\mathcal T}}

\newcommand\sE{{\mathcal E}}
\newcommand\sA{{\mathcal A}}
\newcommand\sF{{\mathcal F}}

\newcommand\sL{{\mathcal L}}

\newcommand\sN{{\mathcal N}}
\newcommand\sK{{\mathcal K}}

\newcommand\sY{{\mathcal Y}}
\newcommand\sH{{\mathcal H}}

\newcommand\sM{{\mathcal M}}

\usepackage{textcomp} 
               
\newcommand\om{\omega}
\newcommand\la{\lambda}
\newcommand\Lam{\Lambda}

\newcommand\e{\epsilon}
\newcommand\s{\sigma}

\newcommand\Ga{\Gamma}

\newcommand\ga{\gamma}
\newcommand\de{\delta}

\newcommand{\CC}{\ensuremath{\mathbb{C}}}
\newcommand{\RR}{\ensuremath{\mathbb{R}}}
\newcommand{\ZZ}{\ensuremath{\mathbb{Z}}}
\newcommand{\QQ}{\ensuremath{\mathbb{Q}}}

\newcommand{\sS}{\ensuremath{\mathcal{S}}}

\newcommand{\NN}{\ensuremath{\mathbb{N}}}
\newcommand{\hol}{\ensuremath{\mathcal{O}}}

\newcommand{\PP}{\ensuremath{\mathbb{P}}}

\newcommand{\FF}{\ensuremath{\mathbb{F}}}

\newcommand{\ra}{\ensuremath{\rightarrow}}

\def\eea{\end{eqnarray*}}
\def\bea{\begin{eqnarray*}}

\DeclareMathOperator{\Sing}{Sing}

\newcommand\dual{\mathrel{\raise3pt\hbox{$\underline{\mathrm{\thinspace d
\thinspace}}$}}}
\newcommand\qe{\ifhmode\unskip\nobreak\fi\quad $\Box$}       

\def\BOX{\hfill\lower.5\baselineskip\hbox{$\Box$}}

\newtheorem{theorem}{Theorem}
\newtheorem{theo}[theorem]{Theorem}
\newtheorem{remark}[theorem]{Remark}
\newenvironment{rem}{\begin{remark}\rm}{\end{remark}}

\newtheorem{prop}[theorem]{Proposition}
\newtheorem{cor}[theorem]{Corollary}

\newtheorem{example}[theorem]{Example}
\newenvironment{ex}{\begin{example}\rm}{\end{example}}

\theoremstyle{definition}
\newtheorem{defin}[theorem]{Definition}
\newenvironment{definition}{\begin{defin}\rm}{\end{defin}}

\newcommand\NR{} 
\def\NR[#1,#2]{%
 node[draw, rounded corners ](#1){#2}
}

\def\red{\color[rgb]{1,0,0}}
\def\verde{\color[rgb]{.2,.5,.1}}

\usepackage{hyperref}

\makeatletter
\def\tagform@#1{\maketag@@@{\ignorespaces#1\unskip\@@italiccorr}}
\makeatother

\newcolumntype{H}{@{}>{\lrbox0}l<{\endlrbox}} 

\begin{document}

\title[Coding theory of Normal surfaces, ]{A  new coding theory,  for normal surfaces,  and   ADE singularities, I }


\author{ Fabrizio Catanese}
\address{Fabrizio Catanese, 
 Mathematisches Institut der Universit\"{a}t
Bayreuth, NW II\\ Universit\"{a}tsstr. 30,
95447 Bayreuth, Germany.}


\email{Fabrizio.Catanese@uni-bayreuth.de,}

\maketitle

\textit{Dedicated to  Wolfgang Ebeling, a gentle spirit and a deep mathematician,  in memoriam.} 

\begin{abstract}
In this article we extend the theory of the binary codes (the strict code $\sK$ and the extended code $\sK'$),  associated to
a projective nodal surface, to a coding theory for   normal surfaces, with special consideration
of the surfaces with ADE (Rational Double Points) singularities.

We define a new theory of generalized labeled codes, establish in the geometric case basic restrictions for the weights
of these codes, and some basic inequality. A crucial  method  that we establish is the  extension
  of the concept of `code shortening'
  to the case of generalized codes: this  is the algebraic counterpart of the geometric notion of a partial smoothing
of the singular points, and leads to the concept of ancestors, which we illustrate through several examples.
\end{abstract}


\tableofcontents

\setcounter{section}{0}

\section{Introduction}

The interplay between coding theory and lattice theory is a central thread in mathematics, as witnessed for instance by 
Wolfgang Ebeling's works \cite{codes-lattices}, \cite{monodromy}, among many other ones (as \cite{conway-sloane}).

In a sense, codes can be viewed  as a finite approximation to lattices, which are  free abelian groups endowed with a 
$\ZZ$-valued symmetric bilinear form. In geometry, codes may be used to describe the saturation of a given sublattice $\sL$  inside a  lattice
$\Lambda$,  which usually for us  is the integral second cohomology group $H^2(S, \ZZ)$ of an algebraic surface (or of an oriented compact 4-manifold).

The advantage of coding theory over a finite field is of being essentially linear algebra, and as such  it has stunning connections
with the classification of configurations of sets of points in finite projective spaces.

For geometric applications, however,  one feels  the need to enlarge the realm of coding theory, replacing codes over
finite fields by more general embeddings  of a finite abelian group inside a  direct sum of labeled 
abelian groups. The geometric underlying idea is to view the first homology group $H_1$ of the smooth part
of a normal complete 
algebraic surface as the  quotient of a direct sum of local homology groups, labeled by a  class of surface 
singularities (for ADE singularities, we just look at holomorphic isomorphism classes, while, in the most general case, we have an equivalence  class, including   the relation generated by  equisingular deformation).

Cumbersome as it may look on the onset, the key feature, as in \cite{nodal}, is to relate partial smoothing
of singularities to an operation in coding theory, called shortening. Whereas, for codes given by a vector
space $\sK$ of functions defined on a set $\Sigma$, shortening is a very simple concept,
amounting to taking the functions with support contained in a subset $\sN \subset \Sigma$, 
in our general theory the concept is more subtle, due to its direct geometric meaning
related to partial smoothings, which may produce several singular points out of a complicated singular popint.
 But the main idea is that, once a singular surface is unobstructed (this means that we may independently achieve
 all the local deformation of the singularities) 
one can define the concept of ancestors, which are roughly speaking the singular surfaces with a maximal set of singularities,
and from which one can determine all other classes of singular surfaces.

We can  then view the stratification of the space of singular surfaces
of a fixed degree $d$ in $\PP^3 := \PP^3_{\CC}$, especially for $d \leq 4$,  as the collection of partial smoothings of
a finite set of ancestors (the case $d=3$ is illustrated in the final section). 

To make  the above discussion more concrete (otherwise it would be hardly understandable), 
I want now to illustrate the above  idea describing the  main motivation for the present work:  the success of binary coding theory for the study of nodal surfaces,
and in particular the following theorem, proven in \cite{nodal}, which determined the irreducible components of the space
of nodal quartic surfaces in $\PP^3$  (the complete classification, in spite of a lot of classical work, \cite{rohn86}, \cite{rohn87}, \cite{jessop1916quartic}, was  still  open).

\begin{theo}\label{Nodal-Quartics}

(I)  The subset of the Nodal Severi variety  $\sF(4, \nu)$ corresponding to  the nodal quartic surfaces 
with a  fixed number $\nu$ of nodes (it must be $\nu  \leq 16$)  and  fixed  binary code
$\sK'$ is smooth and connected, hence irreducible.

(II) The possible codes $\sK'$ 
 appearing are  exactly all the shortenings of the extended  Kummer code $\sK'_{Kum}$,
and are listed as: cases (0), (00), (i), (1), (2) , (II),(III-1), (III-2), (a), (b), (c), (d); hence we have 11 nontrivial codes in their minimal code-embedding.

(III)   For $ \nu \in \{ 0,1,2,3,4,5, 14,15,16\}$, $\sF (4, \nu)$ has exactly one irreducible component, and 
exactly two irreducible components for $ \nu = 6,7,13$, exactly three irreducible components for $ \nu = 8,9, $
exactly four  irreducible components for $ \nu = 11,12, $ exactly five  irreducible components for $ \nu = 10. $

(III) A component of   $\sF(4, \nu)$ is in the closure of another component of $\sF(4, \nu')$
if and only if $\nu \geq \nu'$ and the second code is a shortening of the first.
Hence  the stratification of   $\sF(4, \nu)$ is made of 34 strata, where each stratum is in the boundary of another 
according to   the above Genealogy tree = Kummer nontrivial genealogy for nodal quartic surfaces.
\end{theo}


{\bf Kummer genealogy tree of nodal quartic surfaces.}
\begin{center}
\begin{tikzpicture}[scale=.5]


\node[anchor=west]  at (-0.5,1.5) {\# of Nodes};   

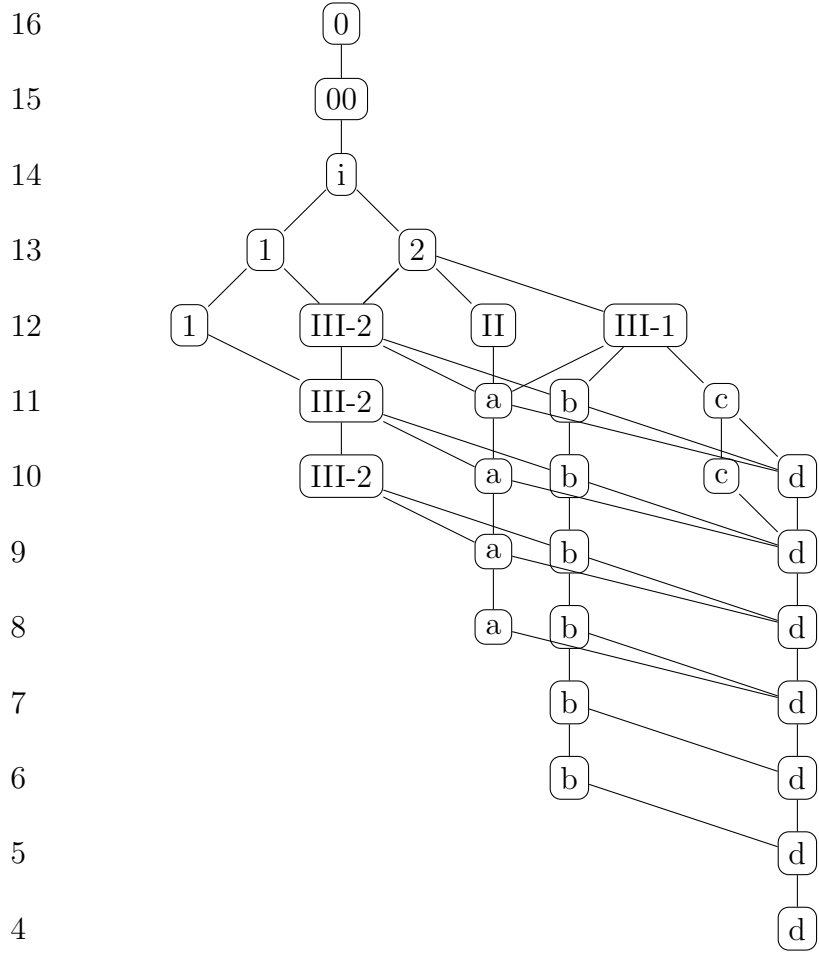
\node[anchor=west]  at (1,0) {16};   
\node[anchor=west]  at (1,-2) {15};
\node[anchor=west]  at (1,-4) {14};
\node[anchor=west]  at (1,-6) {13};
\node[anchor=west]  at (1,-8) {12};
\node[anchor=west]  at (1,-10) {11};
\node[anchor=west]  at (1,-12) {10};
\node[anchor=west]  at (1,-14) {9};
\node[anchor=west]  at (1,-16) {8};
\node[anchor=west]  at (1,-18) {7};
\node[anchor=west]  at (1,-20) {6};
\node[anchor=west]  at (1,-22) {5};
\node[anchor=west]  at (1,-24) {4};

\path 
(10,0)  \NR[N_0,0]
(10,-2) \NR[N_00,00] 
(10,-4) \NR[N_i,i];

\draw (N_0)--(N_00)--(N_i);

\path
(8,-6) \NR[N_1a,1] (12,-6) \NR[N_2,2];

\draw (N_i)--(N_1a); \draw (N_i)--(N_2);

\path
(6,-8) \NR[N_1b,1] (10,-8) \NR[N_III_2a,III-2] (14,-8) \NR[N_II,II] (18,-8) \NR[N_III_1,III-1];

\draw (N_1a)--(N_1b); \draw (N_1a)--(N_III_2a); 
 \draw (N_2)--(N_III_2a); \draw (N_2)--(N_III_2a); \draw (N_2)--(N_II); \draw (N_2)--(N_III_1); 

\path
(10,-10) \NR[N_III_2b,III-2] (14,-10) \NR[N_a11,a] (16,-10) \NR[N_b11,b] (20,-10) \NR[N_c11,c];

\draw (N_1b)--(N_III_2b);
\draw (N_III_2a)--(N_III_2b); \draw (N_III_2a)--(N_a11); \draw (N_III_2a)--(N_b11);
\draw (N_II)--(N_a11);
\draw (N_III_1)--(N_a11); \draw (N_III_1)--(N_b11);\draw (N_III_1)--(N_c11);

\path
(10,-12) \NR[N_III_2c,III-2] (14,-12) \NR[N_a10,a] (16,-12) \NR[N_b10,b] (20,-12) \NR[N_c10,c] (22,-12) \NR[N_d10,d];

\draw (N_III_2b)--(N_III_2c); \draw (N_III_2b)--(N_a10); \draw (N_III_2b)--(N_b10);
\draw (N_a11)--(N_a10);\draw (N_a11)--(N_d10);
\draw (N_b11)--(N_b10);\draw (N_b11)--(N_d10);
\draw (N_c11)--(N_c10);\draw (N_c11)--(N_d10);

\path
 (14,-14) \NR[N_a9,a] (16,-14) \NR[N_b9,b]  (22,-14) \NR[N_d9,d];

 \draw (N_III_2c)--(N_a9); \draw (N_III_2c)--(N_b9);
\draw (N_a10)--(N_a9);\draw (N_a10)--(N_d9);
\draw (N_b10)--(N_b9);\draw (N_b10)--(N_d9);
\draw (N_c10)--(N_d9);
\draw (N_d10)--(N_d9);

\path
 (14,-16) \NR[N_a8,a] (16,-16) \NR[N_b8,b]  (22,-16) \NR[N_d8,d];

\draw (N_a9)--(N_a8);\draw (N_a9)--(N_d8);
\draw (N_b9)--(N_b8);\draw (N_b9)--(N_d8);
\draw (N_d9)--(N_d8);
\path
 (16,-18) \NR[N_b7,b]  (22,-18) \NR[N_d7,d];

\draw (N_a8)--(N_d7);
\draw (N_b8)--(N_b7);\draw (N_b8)--(N_d7);
\draw (N_d8)--(N_d7);

\path
 (16,-20) \NR[N_b6,b]  (22,-20) \NR[N_d6,d];

\draw (N_b7)--(N_b6);\draw (N_b7)--(N_d6);
\draw (N_d7)--(N_d6);

\path  (22,-22) \NR[N_d5,d] (22,-24) \NR[N_d4,d];

\draw (N_b6)--(N_d5);
\draw (N_d6)--(N_d5) -- (N_d4);
  
\end{tikzpicture}
\end{center}

The above main theorem for quartics, Theorem  \ref{Nodal-Quartics}, is based, as already said, on binary coding theory,
but the main new idea, beyond the use of the Torelli theorem for K3 surfaces (see \cite{torelli}) and 
Nikulin's theorems on primitive embeddings
of lattices in the K3 lattice  \cite{nikulin}, is  really 
 the relation between unobstructedness and shortenings (Theorem \ref{thm_d_realized}).
 
 Nikulin' theory of primitive embeddings of lattices offers criteria for existence and unicity of such embeddings
of lattices, and
these results are very useful for the wider  class of nodal polarized K3 surfaces.

 If we fix  a `potential' K3 code $\sK'$, $\sK'$ determines a lattice $\sL'$, candidate to be 
  $\sL^{sat} : = \Lambda \cap \QQ \sL$; and 
 it turns out, using the  results of Nikulin on integral quadratic forms,  that 
there is at most one    primitive embedding of $\sL'$  inside $\Lambda$, and one can then decide whether
such a primitive embedding exists.
\bigskip

In this way the theorem for quartics extends also to all nodal K3 surfaces as follows (this is an abridged version, see \cite{nodal} for
more details).

\begin{theo}\label{mtK3}
The connected components of the 
Nodal Severi variety $\sF_{K3}(d, \nu)$ of nodal K3' s of degree $d$ ($d$ is  any even number) and with $\nu$ nodes are in bijection with the isomorphism classes of their extended codes $\sK'$ (equivalently, of the pair of codes $\sK \subset \sK'' \subset \FF_2^{\nu}$).

Apart from some sporadic cases, which are however obtained via the projection from a node of a nodal K3 surface, we have
several main stream cases, that is, we get all the shortenings  $\sK''$ of five  fixed ancestors:

\begin{enumerate}
\item
If $ d \equiv 0 \ ( mod \ 8) $ of   the K8 code 
(generated by the strict Kummer code $\sK_{Kum}$ and by the characteristic function  of an affine plane $\pi$).
\item
 If $ d \equiv 2 \ ( mod \ 8) $ of the code  
generated by the linear functions
on $\FF_2^4$ and by the quadratic function $x_1 y_1 + x_2 y_2 + 1$. 
\item
If $ d \equiv 4 \ ( mod \ 8) $ and the weight $ t =2$ does not occur,   of the extended Kummer code
$\sK''_{Kum}$.
\item
 If  $ d \equiv 4 \ ( mod \ 8) $ and the weight $ t =2$ occurs,  we get all   the shortenings of the special 
code $\sK'$ occurring for $ \nu=14$, $k'=4$.
\item
 If $ d \equiv 6 \ ( mod \ 8) $  we get all the shortenings $\sK'$ of the two codes
occurring  for $\nu=15$:   the first $\sK'$   is the strict Kummer code, in the second $\sK'$
 there is a weight $4$ vector.
\end{enumerate}

\end{theo}

The above Theorems, yielding five  genealogy trees,  suggest the possibility (dream?)  of understanding  the equisingular components of the space of 
polarized singular K3 surfaces, whose singularities are  Rational Double Points = ADE Singularities,
in terms of suitable ancestors, which should replace the role taken by  the 16-nodal Kummer quartic surfaces
for the case of nodal quartics. This should work for quartic surfaces with ADE singularities without 
the need to consider sporadic cases. In this direction, we plan  in the sequel to this paper to see whether 
it is possible to combine
 the methods of this article with other 
  methods and results (\cite{urabe}, \cite{urabe2} and \cite{gpp}), where some ADE configurations
 on quartic  surfaces and on unpolarized K3 surfaces were  investigated. 

An ancestor is a component which is closed, hence with a maximal set of singularities, 
and each ancestor   determines a genealogy tree of components which have the ancestor in their closure.

Since in the case of nodal surfaces the incidence relation between components (being contained in the closure 
of the other) is determined by the property that the associated binary codes are the second a shortening of the first,
our purpose in this article is to describe a generalized coding theory achieving this purpose for 
 surfaces with ADE singularities.
 And our investigation is directed towards all surfaces with ADE singularities, also those of higher degree.

Our first and main goal in this article will then be  to introduce  a new coding theory for normal surfaces with singularities of a fixed type,
describing  in greater  generality the concepts we use. 

We shall restrict later on  to the case of codes associated to surfaces with ADE singularities, because
these singularities  are amenable
to an easier treatment: here  the  geometric ideas used to describe a more restricted class of shortenings,
which we call geometric driven, lead easily  to a complete description and classification.

We hope that this  first general algebraic step may    be found interesting and useful 
 also for other applications than the geometric ones which motivated us. \footnote{ It was pointed out by Sascha Kurz that codes with alphabet
being a ring have been studied, see \cite{4names}, \cite{3names}; but our notion is quite more general: forgetting 
much of the information about $\sK$, remains a code $$trace (\sK) \subset (G_1 \times G_2 \times \dots \times G_n),$$
where the $G_i$'s are finite abelian groups.}

We illustrate now  the contents of the paper and some of  our main results.

We begin recalling some general results for a normal complex space with isolated singularities, implying that
under suitable assumptions  the 
first homology group of the smooth part is a quotient of the direct sum of the local homology groups
at the singular points, see Theorem \ref{jdg} and Proposition \ref{H_1-pres}.

Then we focus on the local fundamental groups and local homology groups of ADE-singularities, 
which are going to be crucial for our results. 

We describe then, as a warm up,  the general picture of a normal crossing configuration of rational curves on a complex surface, 
and the concept of geometric shortening as the procedure of deleting some of the rational curves from the
configuration. We illustrate the effect on the  local first homology groups, illustrating a few examples.

These examples are to be seen as a key  grasp to understanding the general abstract definitions
which are later given,  of labeled generalized codes, and their shortenings.

We then concentrate on the case where the codes are finite, in this case we can define the concept of the 
weight and refined weight
of a vector, the latter 
 amounts to the number of times where some coordinates are equal and have the same labeling.

In the subsequent  section we define a special class of shortenings, which we call geometric driven, and 
show that for unobstructed surfaces these are the algebraic counterpart to a partial smoothing
of some of the singular points (Theorem \ref{thm_d_realized}). We also establish, in  Theorem \ref{singK3},
the unobstructedness
of K3 surfaces with ADE singularities.

In section 5 we restrict to ADE-codes, and, using the primary decomposition, we prove several
results (Propositions \ref{N=2},  \ref{N=3}, \ref{N>3}, \ref{N=5}) yielding restrictions for the weights of the primary components, namely for the weights
of vectors of prime order, the most important cases being the orders $N=2$, $N=3$, $N=5$, and,
in general,  the almost simple case.

The case $N=3$ reminds us of the work of Barth and Rams \cite{b-9}, \cite{b-8}, \cite{b-r-1},  \cite{b-r-2},
who considered the case of surfaces with $A_2$ singularities, cusps in their terminology, which 
lead to codes $\sK$ over the field $\FF_3 = \ZZ/3$, which they investigated for low degree surfaces. 

Later on, we extend to our general setting, using the binary component of the code, the inequality 
observed  by Beauville in \cite{angers}.

We proceed then with the important step of  establishing in this general context the theory of the extended code $\sK'$, which keeps track
of the polarization of the projective surface $Y$, and whose underlying group is the dual
of the first homology of the complement, inside the smooth locus, of a smooth hyperplane section $H$. 

The extended code, as already illustrated,
plays the key role in the theory, since  it  allows to compute the saturation,
inside the cohomology  $\Lambda:= H^2(S, \ZZ)$ of the minimal resolution $S$ of $Y$, 
of the sublattice $\sL$ generated by the exceptional curves and by a smooth hyperplane section (Theorem \ref{codehomology2}). 

We provide then some restrictions for the weights also for vectors of the extended code $\sK'$.

  In the last section we conclude showing  some  examples, describing for instance the surfaces which  are  the ancestors for the cubic surfaces 
 in $\PP^3$, their codes, and then indicating some ancestors for non-polarized K3 surfaces;
 we hope, in the sequel to this article,  to find more substantial applications.

 \bigskip
 
{\bf Acknowledgements:} thanks to Matteo Penegini for pointing out some results of his joint paper \cite{gpp}.
A good part of this work was done as   KIAS Scholar  in November 2023: thanks to Jong Hae Keum
for hospitality and useful conversations.

 \bigskip

  \bigskip
 
 \bigskip

\section{Local to global homology  of Normal compact complex spaces  $X$ with isolated singularities}

Let $X$ be a  normal compact complex space  of dimension $n$ with isolated singularities,   let 
$$s : = s(X): = | Sing(X)|$$ be the number of its singular points, and  denote further by $\Sigma$ the  singular set  of $X$
$$ \Sigma : = Sing (X) = \{ P_1, \dots, P_{s}\}.$$

The topological structure of the smooth part $X^* : = X  \setminus Sing(X)$ is an important invariant of $X$, which
is not changed by equisingular deformations (see \cite{wahl}); in particular the fundamental group $\pi_1 (X^*)$
and its Abelianization, the first homology group $H_1(X^*, \ZZ)$, are two basic such invariants.

Let $S$ be a minimal resolution of singularities of $X$, and let $H$ be a smooth divisor in $X$ not passing through the singular points 
of $X$ (in the case where $X$ is projective, $H$ shall be chosen to be a very ample  divisor). 

Recall  the following result  of  \cite{modulispace}, pages  489-490, which shall be applied in two important cases:
\begin{itemize}
\item
i) the case where $S$ is a minimal resolution of singularities of $X$, and $D$ is the 
reduced exceptional divisor $E$ of $p : S \ra X$;
\item
ii) ditto with $D$ the union of $E$ with the inverse image of $H$.
\end{itemize}

\begin{theo}\label{jdg}
Let $S$ be a smooth compact complex manifold of complex dimension $n$, and let $D = D_1 \cup \dots \cup D_{\nu}$ 
be a reduced divisor,
where each $D_i$ is irreducible. 

Then the surjection  $ H_1( S \setminus D , \ZZ) \twoheadrightarrow H_1( S  , \ZZ)$ has kernel equal to the cokernel of
$ \rho : H^{2n-2} (S, \ZZ) \ra  H^{2n-2} (D, \ZZ) = \oplus_1^{\nu} \ZZ [D_i]$.

For $n=2$, $\rho(L) = \sum_i (L \cdot D_i) [D_i]$.
\end{theo}

Given the singular point $P_i$, let $U_i$ be a neighbourhood  of $P_i$ obtained by  intersecting $X$ with a small 
Euclidean ball with centre $P_i$,  set $U_i^* : = U_i \setminus \{P_i\}$
and consider  the local fundamental group
$$ \pi_{1, loc} (X_{P_i} ): = \pi_1( U_i^*) = : \Ga_i ,$$ 
whose abelianization is called the local homology group 
$$ H_1(P_i) : = H_{1, loc} (X_{P_i} , \ZZ): = H_1( U_i^*, \ZZ) = ( \Ga_i )^{ab}.$$ 

We focus on the natural local to global homomorphism 

\begin{equation}\label{global}
gl : [ \oplus_{P_i \in Sing(X)}  H_1(P_i) = H_{1, loc} (X_{p_i} , \ZZ) ] \ra H_1 (X^*, \ZZ) =: H_1.
\end{equation} 

This homomorphism $gl$ will be particularly interesting in the case where the local homology groups are finite and 
also $H_1 (X^*, \ZZ)$ is finite.

The following proposition is proven in \cite{nodal}.
\begin{prop}\label{H_1-pres}
 (i) Let $X$ be a  normal compact complex space  of dimension $n \geq 2$ with isolated singularities,
 with a normal crossing divisor resolution of singularities $S$ having $ H_1 (S, \ZZ)=0$.
 
 Then the  local to global homomorphism $gl$ (see \eqref{global}) is surjective, and its image $H_1 : = H_1 (X^*, \ZZ)$ is isomorphic to $Coker (\rho)$,
 hence fits into an exact sequence
 $$ H^{2n-2} (S, \ZZ) \ra  H^{2n-2} (D, \ZZ) = \oplus_1^{\nu} \ZZ [D_i] \ra H_1 \ra 0.$$
 
(ii) The same is true  if $X$ admits a smoothing (a flat deformation over the disk in $\CC$ with general fibre a smooth manifold $Y$),
 such that the general fibre $Y$ is simply connected, and the Milnor fibres are 1-connected. This happens for instance if $X$ is a complete intersection in $\PP^N$.
 
  In this latter case we have the  stronger result that  
$\pi_1 (X^*)$ is normally generated by a quotient of the free product $\Ga_1 * \Ga_2 * \dots *  \Ga_{\nu}$ of the local fundamental groups at the singular points.

\end{prop}
\begin{proof}

Concerning (i), we apply Theorem \ref{jdg} by slightly changing the notation.
Namely, for each $P \in Sing(X)$, we let $E_P$ be the exceptional divisor, so that $D $ is the disjoint union of the
divisors $E_P$.

Then, as in \cite{modulispace} loc. cit. $H_1 (X^*, \ZZ) = H_1 (S \setminus D, \ZZ)= H_1 (S \setminus V, \ZZ)$,
where $V$ is a tubular neighbourhood of $D$.

Since the kernel of the surjection $\pi_1(S \setminus D) \ra \pi_1(S)$ is normally generated by small
geometrical  loops $\ga_i$ around the (irreducible) divisors $D_i$, it follows that 
$H_1 (S \setminus D, \ZZ)$ is generated by the homology classes $[\ga_i]$ of these loops. These classes $[\ga_i]$
come from $H_1 (V_P  \setminus E_P, \ZZ) = H_{1, loc} (X_{P})$.

In turn, by Lefschetz duality, $H_1 (S \setminus V, \ZZ) \cong H^{2n-1} ( S \setminus V, \partial  V, \ZZ)$
which by excision equals $H^{2n-1} ( S, \bar{  V}, \ZZ) \cong H^{2n-1} ( S , D, \ZZ)$
which equals then the Cokernel of $\rho : H^{2n-2} ( S,  \ZZ) \twoheadrightarrow  H^{2n-2} ( D,  \ZZ)$.

For the proof of (ii),  observe that $X$ deforms to a smooth manifold  $Y$, with  $\pi_1 (Y) = 1$; this is a consequence of
  Lefschetz' theorem, if $X$ is a complete intersection of dimension $n \geq 2$.

Define $F$ to be the union of the Milnor fibres $F_i$, together with a tree connecting the base point  to each  
Milnor fibre with a respective segment (in the case where $X$ is a complete intersection that  $\pi_1(F) = 1$ 
follows by Hamm's extension \cite{hamm} of Milnor's theorem \cite{milnor}
stating   that $F_i$
is homotopically equivalent to a bouquet of spheres). 

The first van Kampen 's theorem says that $ 1 = \pi_1 (Y') $ is the quotient of the free product $\pi_1 (Y^*) * \pi_1(F) = \pi_1 (Y^*)$
by the subgroup normally generated by $\pi_1 (F^*)$, where $F^* : = F \cap Y^*$. Now, by construction, $\pi_1 (F^*)$
is the free  product $\Ga_1 * \Ga_2 * \dots *  \Ga_{\nu}$, hence the desired assertion follows.

When we abelianize, we obtain a surjection of $\Ga_1 * \Ga_2 * \dots *  \Ga_{\nu}$ onto  $H_1(Y^*, \ZZ)$, which therefore factors through a surjection
$$\Ga^{ab} _1 \times  \Ga^{ab} _2 \times \dots \times  \Ga^{ab}_{\nu}  \twoheadrightarrow H_1(Y^*, \ZZ).$$

\end{proof}

\subsection{The case of normal surfaces} 
Let us  consider first  the simplest possible case, where $X$ is a nodal surface, that is, all singular points of $X$ are
{\bf nodes}, hypersurface singularities of multiplicity two with nondegenerate Hessian (hence locally biholomorphic to the singularity
$y_1 y_2 -  y_3^2  = 0$ inside $\CC^3$).

A node is the quotient  $\CC^2 / (\pm 1)$, since, if $u_1, u_2$ are local coordinates on $\CC^2$, the quotient is embedded by $y_i : = u_i^2,
\ i=1,2$ and by $y_3 : = u_1 u_2$. 

In particular, the local fundamental group is then $\ZZ/2$.

\begin{example}\label{2}
If $X$ is a nodal surface in $\PP^3$, with $ \Sigma = Sing (X) = \{ P_1, \dots, P_s\},$ then $H_1(X^*, \ZZ) \cong (\ZZ/2)^k$,
and there is a natural surjection $\oplus_1^{s}  (\ZZ/2) \twoheadrightarrow H_1(X^*, \ZZ).$
\end{example}

In order to justify our forthcoming definitions of generalized codes let us look at the example of 
the binary codes associated to nodal surfaces.

\begin{example}
Let $X$ be a nodal surface in $\PP^3$: then its binary code $\sK \subset (\ZZ/2)^{s}$ is defined as the 
image of the injective  map dual to
 the surjection $$\oplus_1^{s}  (\ZZ/2) \twoheadrightarrow H_1(X^*, \ZZ).$$
\end{example}

\begin{example}\label{general-normal-surface}
Take  more generally  $X$ to be a normal surface  whose  singularities are rational double points (ADE singularities), 
and whose minimal resolution
$S$ is simply connected.  Since $X$ deforms differentiably to $Y = S$, see for instance  \cite{cat-sympl}, saying that $S$ 
is simply connected is equivalent to say that $X$ has a simply connected smoothing.

In this case again $gl$ is surjective. 

Moreover, since $S$ is simply connected, then $$ \Lambda : = H^2(S, \ZZ)$$
is a unimodular lattice, that is, a  free abelian group endowed (by Poincar\'e duality) with   a unimodular bilinear form.

Set $\sL : = \oplus_1^{\nu} \ZZ D_i $: since the intersection matrix $(D_i \cdot D_j)$  is negative definite, then 
$ \sL \hookrightarrow \Lambda$.

\end{example}

\subsection{Local homology groups of ADE-singularities, also called Rational Double Points and Kleinian Singularities}
Following for instance the notation of \cite{cetraro}, pages 71-74, we consider the Kleinian singularities 
$$ X_x = \CC^2/G, \ G < SU(2, \CC).$$

Let $G' < \PP SU(2, \CC) \cong SO(3)$ be the image of $G$, and let $\hat{G}$ be the full inverse image of $G'$ through the exact sequence
$$ 1 \ra \{ \pm 1\} \ra SU(2, \CC) \ra \PP SU(2, \CC) \ra 1,$$
so that 
$$ 1 \ra \{ \pm 1\} \ra \hat{G}\ra G' \ra 1.$$

Clearly, $\pi_{1, x} (P) \cong G$, while $H_1(x) \cong (G)^{ab}$, according to the following table \footnote{about which we shall provide more details in the ensuing subsections.}
($H_1(x) \cong (G')^{ab}$ except for the case of $A_n$ with $n$ odd).

\begin{itemize}
\item
 For a singularity of type $A_n$,  $G = H_1(x) = \ZZ/(n+1)\ZZ $.
\item
For a singularity of type $D_{2m+1}$, $G $ is the binary dihedral group,  $G'$ is the dihedral group, but
with a shift of index,
 $$G = \hat{D}_{2m-1},\  G' =  D_{2m-1}, \ H_1(x) = \ZZ/ 4 \ZZ .$$
\item
 For a singularity of type $D_{2m}$,  $$G = \hat{D}_{2m-2}, \ G' =  D_{2m-2}, \ H_1(x) =(\ZZ/2\ZZ )^2$$
\item
 For a singularity of type $E_{6}$ $$G = \hat{\mathfrak A_4}, \ G' = \mathfrak A_4, \ H_1(x) =\ZZ/3 \ZZ .$$
\item
 For a singularity of type $E_{7}$ $$G = \hat{\mathfrak S_4}, \ G' = \mathfrak S_4, \ H_1(x) =\ZZ/2 \ZZ .$$
\item
For a singularity of type $E_{8}$ $$G = \hat{\mathfrak A_5}, \ G' = \mathfrak A_5, \ H_1(x) = 0.$$
\end{itemize}

\subsection{Local homology of normal surface singularities and groups of normal crossing curve configurations}

Let $x \in X$ be (the germ of) a normal surface singularity, and $p : S \ra X$ a relatively minimal (global) 
normal crossing resolution of singularities, so that $p^{-1}(x)$ is a connected normal crossing divisor \footnote{ When dealing with a singularity of type $D_n$,
we shall switch notation to $ D = \cup_i C_i$.}
$$ D = \cup_i D_i$$
such that the intersection matrix $(D_i \cdot D_j)$ is negative definite (this implies that $K_s \cdot D_i \geq 0$,
as noted in \cite{zappa}, from which we borrow for the following facts).

Let $T= T(D)$ be a good tubular neighbourhood of $D$, so that $ (T\setminus D ) \cong X \setminus \{x\}$.

The local fundamental group $\pi_1 ( X \setminus \{x\}) = \pi_1 ( T\setminus D )$ can be computed
as a group with generators and relations, see \cite{zappa} pages 140 and following.

In view of the surjection $ \pi_1 ( T\setminus D ) \twoheadrightarrow   \pi_1 (  D )$ if we assume that the local first homology group
$H_1(x) : = H_1(  X \setminus \{x\}, \ZZ)$ is finite, it follows that each $D_i$ is a smooth rational curve,
and we have a graph of rational curves with zero  first homology (a tree).

Under these two assumptions, we have a presentation, where we associate a generator $\ga_i$ to each component $D_i$,
 a generator $\ga_{ij}$ for each point $D_i \cap D_j$, and we set $- m_i : = D_i^2$:

$$ \pi_1 ( T\setminus D )  = \langle \ga_i,  \ga_{ij} | \ga_{ij} = \ga_{j}, \ga_i^{m_i} = \prod_j \ga_{ij}, [\ga_i,\ga_{ij}]=1\rangle .$$
We have $$H_1(x) = \pi_1 ( T\setminus D )^{ab} = (\oplus \ZZ \ga_i ) / (\oplus \ZZ \ga_j )^{\vee},$$
in other words, $H_1(x)$ is the cokernel of the intersection matrix. 

Hence it is a finite abelian group  if $(D_i \cdot D_j)$ is negative definite.

The Rational Double Points, also called  ADE singularities, yield a case where we have a graph of smooth rational curves
with $m_i=2, \forall i$.

The above calculation yields not only the local homology group $H_1(x)$, but also a set of generators $\ga_i$.

\begin{enumerate}
\item
$A_n$ corresponds to a linear  tree of $n$ rational curves, and here we have as generator $\ga_1$, and the relations
amount to $\ga_j = j \ga_1$, and to  $ 2 \ga_n = \ga_{n-1} \Leftrightarrow (n+1) \ga_1=0.$
\item
$D_n$ corresponds to a  linear tree of $(n-1)$ rational curves, plus $\ga_n$ intersecting $\ga_{n-2}$ 
(this means, by abuse of notation, that the curve $C_n$ corresponding to $\ga_n$ intersects $C_{n-2}$):
here  the relations
amount to $\ga_j = j \ga_1$, for $j \leq n-2$,  and to  $$ 2 \ga_n = 2 \ga_{n-1} = \ga_{n-2} = (n-2) \ga_1,
(n-1) \ga_1 =\ga_{n} + \ga_{n-1}, $$ $$ \Leftrightarrow \ga_{n-1} = \ga_1 + \ga_n, 2 \ga_1 =0, 2 \ga_n = (n-2) \ga_1,$$
hence we get $H_1(x) = \ZZ/4 $ for $n$ odd, with  generator $\ga_n$, and $H_1(x)= \ZZ/2 \oplus \ZZ/2$ for $n$ even,
with generators $\ga_1, \ga_n$.
\item
For $E_6$ we get a  linear tree of five rational curves, corresponding to $\ga_1, \ga_2, \ga_3, \ga_5, \ga_6$, 
and $\ga_4$ intersecting $\ga_3$.

 We get $H_1(x)= \ZZ/3$ with generator $\ga_1= \ga_5$, and moreover  $\ga_2 = \ga_6= 2 \ga_1$,
$\ga_3 = \ga_4=0$.
\item
For $E_7$ we get a  linear  tree of six rational curves, corresponding to $\ga_1, \ga_2, \ga_3, \ga_4, \ga_5, \ga_6$, 
and $\ga_7$ intersecting $\ga_4$. 

We get $H_1(x)= \ZZ/2$ with generator $\ga_1= \ga_7 = \ga_3$, and moreover $\ga_2 =\ga_4= \ga_5= \ga_6= 0$.
\item
For $E_8$ we get a  linear  tree of seven rational curves, corresponding to $\ga_1, \ga_2, \ga_3, \ga_4, \ga_5, \ga_6, \ga_7$, 
and with $\ga_8$ intersecting $\ga_5$. We get $H_1(x)= 0.$

\end{enumerate}

Looking at the $p$-primary components, we see that for $ p>3$ we get only a contribution from $A_n$, with $ p | (n+1)$
(hence $n \geq 4$), for $p=3$ only $A_{3k-1}$ and $E_6$, for $p=2$: $A_{2m+ 1}$ , $D_n$,  $E_7$.

\subsection{Elementary geometric shortenings}
If we have as above a global normal crossing divisor $D = \cup D_i$ with tubular neighbourhood $T= T(D)$,
we consider the operation of deleting one curve $D_i$ from $D$, obtaining $D' := D(i)  : = \cup_{j \neq i} D_j$, a union of connected 
normal crossing divisors. 

We set  $T' : = T \setminus D'$, which contains a union $T(i)$ of tubular neighbourhoods
of the components of $D' = D(i)$.

Assume now that all the curves $D_i$ are smooth rational.

Clearly $\pi_1(T \setminus D) $ surjects onto $\pi_1(T \setminus D') $ and the  kernel is normally generated
by $\ga_i$, hence the same holds true for their abelianizations,
hence $H_1(x) : = H_1 (T \setminus D)$ surjects onto $H_1' (x) := H_1(T \setminus D') $.

On the other hand, also $ H_1( T(i) \setminus D(i))$ surjects onto $H_1' (x) := H_1(T \setminus D') $ by the Mayer-Vietoris exact sequence, and because  we are adding a disk bundle over the punctured curve $D_i \setminus (\cup_{j\neq i} D_j)$,
whose first homology is generated by the elements $\ga_{ij}$, which are in the image of  $H_1( T(i) \setminus D(i))$.

Therefore  $H'_1(x)$, a quotient of $H_1(x)$, is also a quotient of $H_1( T(i) \setminus D(i))=:\oplus_y H'_1(x,y)$, where 
$y$ runs through the connected components of $ T(i) \setminus D(i)$ (the cumbersome notation is clear in the case where $D$
is the exceptional normal crossing divisor in the resolution of a singular point $x$ of a normal surface, and the points $y$
are the singular points of a partial smoothing of the singularity $x$).

\begin{ex}\label{A_n}
Consider the case of $A_n$-singularities, $$A_n : = \{ z^2 - x^2 + y^{n+1}=0\}.$$

The semiuniversal deformation of the singularity consists of the surfaces of equation 
$$ z^2 - x^2 + P(y) =0, \ P = y^{n+1} + a_{n-1} y^{n-1} + \dots + a_0,$$
whose singular set is the set $\{ z=x= P(y) = P_y(y)=0\}$, hence consists of the points $(0,0,\eta)$, where $\eta$ is a multiple
 root of $P(y)$.
 
 The simplest (the elementary) deformations are of the form $P(y) = y^{m+1} (y-t)^{n-m}$.
 
 Then the $A_n$-singularity deforms to an $A_m$ plus an $A_{n-m-1}$-singularity. Taking the minimal resolution, the outcome
 is to take the $A_n$-diagram and delete the curve with index $m+1$.
 
 In terms of fundamental groups, we take the quotient of $H_1(x)  \cong \ZZ/(n+1)$ by the relation $\ga_{m+1}=1$:
 since  $\ga_{m+1} = \ga_1^{m+1} = \ga_n^{n-m}$, and $\ga_n  \ga_1^n$, we get 
 $$H'_1(x)  \cong \ZZ/(m+1, n-m) = \ZZ/d,$$
 where $d : = GCD (m+1, n-m)$.
 
 Whereas the groups $H'_1(x,y)$ are just $\ZZ/(m+1) $ and $ \ZZ/(n-m)$, whose direct sum  
 $\ZZ/(m+1) \oplus \ZZ/(n-m)$ obviously maps onto $ \ZZ/d$, with kernel 
$ \cong \ZZ/ LCM(n-m, m+1)$.

\end{ex} 

The above  example will illuminate later on the concept of geometrically driven  elementary shortenings.

In the case of $D_n$ an elementary geometric shortening leads either to $D_m$ plus $A_{n-m-1}$,
or to $A_{n-1}$, or to $A_1 \cup A_1  \cup A_{n-3}$.

The case of $E_5, E_7, E_8$ is also fun, $E_6$ can go by deleting the middle curve to $A_2 \cup A_2 \cup A_1 $,
correspondingly $(\ZZ/3)^2 \oplus (\ZZ/2) $ maps onto $  0 = H'_1(x)$.

\section{Generalized codes and their shortenings}

\subsection{Duality for abelian groups} Let $A$ be an abelian group. 

Then 
we define the  {\bf Weakly dual abelian group}\footnote{The Pontrjagin dual is defined instead as $Hom ( A, \RR/\ZZ)$, for any locally compact abelian group $A$,  and biduality holds.}
as:
\begin{equation}\label{dual1}
A^* : = Hom ( A, \QQ/\ZZ \oplus \QQ).
\end{equation}
If $A$   is \ finitely  generated, then we define its {\bf dual} as 
\begin{equation}\label{dual2}
A^{\vee} : = Hom ( A, \QQ/\ZZ \oplus \ZZ),  
\end{equation}
observing that, if $A$ is finite, then, using the notation of  \cite{godement}, page 6 \footnote{loc. cit. contains the incorrect claim that every cyclic group embeds in $\QQ/\ZZ $. }
\begin{equation}\label{dual3}
A^{\vee}  = \hat{A} : = Hom ( A, \QQ/\ZZ ).\\
\end{equation}

Observe that we have also a notion of duality for homomorphisms, 
and that,  if $A$  is finitely generated, then 
$$ A \hookrightarrow (A^{\vee} )^{\vee} ,$$
and indeed {\bf biduality}
holds: $$ (A^{\vee} )^{\vee} \cong A.$$ 

\subsection{Generalized Codes} We start with some quite general definitions.

\begin{defin}
Let $\sC$ be a finite {\bf set of labeled finitely generated  abelian groups}, that is, $\sC$ a finite set, given 
together with a map to the category  of finitely generated  abelian groups: to each $j \in \sC$ is associated a
 finitely generated abelian group $V_j$ (and possibly $V_j \cong V_i$ for $i \neq j$).

Then a {\bf $\sC$-labeled-generalized code} consists of 

(1) a finite set $\Sigma$,

(2) a map $h : \Sigma \ra  \sC$ which,  for shorthand notation,
shall be indicated as the association to $x \in \Sigma$ of an abelian group 
$$H_1(x) : = V_{h(x)}  \in \sC,$$

(3) of a surjection $\ks  : \oplus_{x \in \Sigma} H_1(x) \twoheadrightarrow H_1$,

hence (equivalently, since $\sC$ consists of finitely generated abelian groups) also 

(3') of an injective homomorphism 
$$   \ks^{\vee}: H_1^{\vee}  \hookrightarrow \ \oplus_{x \in \Sigma} H_1(x)^{\vee},$$
  and  we set then  $$ \sK : = Im (\ks^{\vee}).$$

(4) Two such codes $(\Sigma, h, \ks)$,  $(\Sigma', h', \ks')$ are equivalent if there is a bijection
 $\phi : \Sigma \ra \Sigma'$ such that $h = h' \circ \phi$,  and carrying $Ker (\ks)$ to $Ker (\ks')$ (hence inducing an isomorphism of the cokernels $Coker (\ks) \cong Coker (\ks')$, and, dualizing, carrying  $\sK$ to $\sK'$).
\end{defin}

\begin{remark}
(i) When defining the equivalence of codes, it is convenient, as in the case of vector codes,
not to allow automorphisms of $H_1(x)$.

In our case, a very good reason is that if we take an $A_n$  tree of smooth rational curves $E= E_1 \cup \dots \cup E_n$
contained in a smooth complex surface, and with $E_i^2=-2$, then $\sT(E)$ being a tubular neighbourhood of $E$,
then $\pi_1 (\sT(E) \setminus E) = \ZZ / (n+1)$, yet the graph produces two standard generators, $\ga_1$ and $\ga_n$,
the inverse of each other. In fact, as we saw,  the (more general) algorithm of Lemma 5 of  \cite{zappa} shows that $\ga_m = \ga_1^m$,
and that $ 1 = \ga_1^{n+1}$. But, in the case where  $(n+1)$ is prime, then each $\ga_m$ is a generator.

(ii) In the case of ADE singularities one may allow automorphisms which come from an
automorphism of the exceptional tree, for instance for $A_n$, to allow exchanging $\ga_1$ with $\ga_{n}$.
This might be useful later on for classification, but allowing  it right away might make the general theory more complicated.
\end{remark}

\begin{defin}
A {\bf naive generalized code} is obtained in the special case where $V_i \cong V_j$ implies 
$ i=j \in \sC$.

We have a {\bf finite generalized  code} (labeled or unlabeled) if  each $V_j$ is a finite abelian group.
\end{defin}

\begin{remark}
One can associate a naive generalized code to a labeled generalized code, by a sort of forgetful map.

It suffices, for each maximal set of mutually  isomorphic abelian groups $V_j$, to choose an isomorphism
of each $V_j$ with a fixed group $V'_j$.

\end{remark}

\subsection{Finite generalized Codes and their primary decomposition}
In this case all the abelian groups occurring in
$$\ks  : \oplus_{x \in \Sigma} H_1(x) \twoheadrightarrow H_1$$
are finite and we can take their primary decomposition, for instance we write
$$ H_1 = \oplus_{p= prime} H_{1,p}.$$

\begin{prop}\label{primary}
A finite generalized code is a direct sum of its $p$-primary components, 
$$ \sK = \oplus_p  \sK_p,$$
$$ \sK_p : =  Im (\ks^{\vee}_p), \ \ks^{\vee}_p: H_{1,p}^{\vee}  \hookrightarrow \ \oplus_{x \in \Sigma} H_{1,p}(x)^{\vee} .$$
Two finite generalized codes are equivalent if and only if their $p$-primary components
are equivalent.
\end{prop}
\begin{proof}
This follows since the primary decomposition is unique, indeed  every homomorphism $\psi :A \ra B$,
where $A$ is $p$-primary, and the $p$-primary component of $B$ is 0, is equal to 0.

\end{proof}

\subsection{Weights of finite generalized Codes}We define now the code vectors, and their weights, as follows.

\begin{defin}
Define a {\bf code vector} as an element 
$$ v \in \sK \cong H_1^{\vee} \hookrightarrow \ \oplus_{x \in \Sigma} H_1(x)^{\vee} .$$

Since $ v : H_1 \ra \QQ/\ZZ$, $ v (H_1)$ is a finite subgroup of $\QQ/\ZZ$, hence of the cyclic group 
$( \frac{1}{M} \ZZ )/ \ZZ \cong \ZZ/M \ZZ$, where  $M \in \NN$ is the exponent of the group $H_1$.

We define the {\bf order} of $v$ as the positive integer $N$ such that $ Im (v) =  (\frac{1}{N} \ZZ )/ \ZZ $.

For each $x$ we get the composition 
$$v_x : H_1(x) \ra H_1 \ra ( \frac{1}{N} \ZZ )/ \ZZ  \subset \QQ/\ZZ,$$ and,
letting,  for $ j \in \sC$, the corresponding abelian group be  $V_j$,
we define  

\begin{itemize}

\item
 the {\bf refined weight} as  the number of times, for $ j \in \sC$, that a given nonzero linear form  $\xi : V_j \ra \QQ/\ZZ$ occurs:
$$ w(v; j, \xi)  : = | \{ x | h(x) = j, \ v_x = \xi \}|.$$ 
\item
The {\bf label weight} as
$$ w(v; j)  : = | \{ x | h(x) = j, \ v_x \neq 0 \}|.$$
\item
The {\bf Hamming weight } $w(v)$ as
\begin{equation}\label{weight}
w (v) : = \sum_{j, \xi \neq 0} w(v; j, \xi) = \sum_j  w(v; j) .
\end{equation}
\end{itemize}
Observe that, if $V_j = \ZZ/ N_j \ZZ$, then $\xi$ is determined by $\xi(1)  \in \ZZ/N \ZZ$.
  
\end{defin}

\subsection{Labeled generalized  code shortenings}

Given a labeled generalized code 

$$\ks  : V: = \oplus_{x \in \Sigma} H_1(x) \twoheadrightarrow H_1,$$

we define shortenings via elementary shortenings.

\begin{defin}
(1) An elementary shortening associated to $z \in \Sigma$ is obtained replacing 
$H_1(z)$ by a $\sC$-labeled generalized code 
$$\ks''  : V'': = \oplus_{y \mapsto z } H'_1(z,y) \twoheadrightarrow H'_1(z),$$
where $H'_1(z)$ is a quotient of $H_1(z)$,
and replacing $H_1$ by the quotient $H'_1$ of $H_1$ by the image of $M_j : = Ker ( H_1(z) \twoheadrightarrow H'_1(z))$
inside $H_1$.

Clearly  we have
$$ k' : V' : = (\oplus_{y \mapsto z } H'_1(z,y)) \bigoplus (\oplus_{x \in \Sigma, x\neq z} H_1(x)) \twoheadrightarrow 
H'_1(z) \bigoplus (\oplus_{x \in \Sigma, x\neq z} H_1(x))
\twoheadrightarrow H'_1. $$

(2) We say that we have a full elementary shortening if $H'_1(z)=0$.

(3) Every shortening is defined as a composition of elementary shortenings.

(4) We say that an elementary shortening is driven if, for each $j\in \sC$,  $\ks'' $
belongs to a list prescribed a priori (as in the case of   geometric driven shortenings).
\end{defin}

\begin{remark}
(i) Observe that, dually, $$\sK' = (H'_1)^{\vee} = (H_1)^{\vee} \cap (V')^{\vee} = \sK \cap (V')^{\vee}.$$

(ii) Using a forgetful map we can also define the shortening of a naive generalized code.

\end{remark}

\subsection{Length of generalized codes?}
The definition of length can be given in several ways, the most natural one  being 
to define the length as the cardinality of $\Sigma$ (alternatively, one can sum, over $x \in \Sigma$,
the number of summands for a canonical decomposition of $H_1(x)$ (a decomposition with cyclic summands,
be it the Frobenius decomposition, or the primary decomposition).

This naive  notion has the disadvantage, as we easily infer, that the length does not decrease by a shortening,
a fact which looks like a  contradiction in words.

For this reason we shall define in the next section the notion of rank, for the geometric codes associated to
normal surfaces: these do indeed decrease via a shortening.

\section{Plumbing of curves and  geometric driven shortenings of codes of normal surfaces, ADE generalized codes}

\begin{defin}
(i) ADE generalized codes are the labeled  generalized codes where $\sC$ is the class of ADE singularities, and
to each singularity we associate the local homology group $H_1(x) $ (given together with a set of generators corresponding to
the exceptional divisors of the minimal normal crossing resolution).

(ii) A more general class is the class of the codes of normal surfaces with rational singularities, 
where the labeling is given by the configuration of exceptional curves in a minimal normal crossing resolution
(the curves are  rational, and the self intersection numbers are part of the datum of the configuration). In fact, the isomorphism classes of ADE singularities are determined by their configurations (for them, the self intersection numbers are $= -2$).

(iii)  elementary geometric driven shortenings correspond to deleting an irreducible curve $D_i \leq D$ from the normal crossing divisor $D$.

(iv)  geometric driven shortenings correspond to deleting a certain number of the irreducible curves $D_i$;
then, contracting  the remaining curves,   one gets a normal surface with a finite number of singularities,
which are ADE singularities  if we started from an ADE singularity.
\end{defin}

\begin{defin}\label{rank}
Given a code  of a normal surface,
where the labeling is given by the configuration of exceptional curves in a minimal normal crossing resolution,
the {\bf rank} of the code is given by the number of exceptional curves. \footnote{For an ADE singularity
the index equals this number, e.g. $n$ for $A_n$, $8$ for $E_8$.}

\end{defin}

\begin{remark}\label{partial-smoothing}
Given the dual graph of a connected normal crossing configuration $D$ of smooth rational curves $D_i$, and
self intersection numbers $D_i^2 = - m_i \in \ZZ$, we can construct, by plumbing,  a smooth (non compact) complex surface $T$ containing the
configuration. If the intersection matrix $ D_i \cdot D_j$ is negative definite, by Grauert's theorem \cite{gra}
we can contract  $D$ to a point $x$, obtaining a normal (non compact) complex surface $Y$ with a singular point $x$,
if $ \pi_1 (T \setminus D) \neq 1$, as shown by Mumford \cite{mum1}.

The isomorphism class of the singularity may vary, but for us the labeling is given by the configuration $D$ (including the multiplicities),
which is an invariant for equisingular deformation. One can of course also consider more general configurations.  
\end{remark}

 Geometric driven shortenings replace an ADE singularity by several other ones, yielding an ADE code with quotient 
$H'_1(x)$ which is a quotient of the local homology $H_1(x)$.

\begin{remark}\label{partial-smoothing}
As shown in the particular case of example \ref{A_n},  a  geometric driven shortening corresponds to a partial
smoothing of an ADE singularity, as shown by Burns-Wahl \cite{burns-wahl}, and as we shall discuss in  the next subsection.
\end{remark}

\subsection{Code shortenings and unobstructedness} 

The condition 
of unobstructedness means that all the local deformations of the singular points of $X$ can be simultaneously achieved by
a global deformation of $X$.  

More precisely, unobstructedness means that the deformations of $X$ have a submersion onto
the space of local deformations of the singularities, so that one can obtain  independent smoothings
of all the singular points.
New results in this direction have been obtained by 
  Dimca and Kloostermann \cite{dimca} \cite{remke}.

In fact, (see for instance \cite{wahl1}, \cite{cime}, \cite{sernesi}) for a normal surface $Y$  in $\PP : = \PP^3$, we have (by definition) the exact sequence 
$$ 0 \ra N^*_Y \ra \Omega^1_{\PP} \otimes \hol_Y \ra  \Omega^1_Y \ra 0,$$ 
where the conormal sheaf $N^*_Y \cong \hol_Y(-d)$ s generated by the differential of $F$, where $F$ is the
polynomial equation (of degree $d$) of $Y$.

Dualizing with respect to $\hol_Y$, we get an exact sequence 
$$ 0 \ra \Theta_Y : = \sH om ( \Omega^1_Y, \hol_Y) \ra  \Theta_{\PP} \otimes \hol_Y \ra N_Y \cong \hol_Y(d) \ra \sE xt^1( \Omega^1_Y, \hol_Y) \ra 0,$$ 
which can be split into two short exact sequences, the second one being:
$$ 0 \ra N'_Y \ra N_Y \cong \hol_Y(d) \ra  \sE xt^1( \Omega^1_Y, \hol_Y) \ra 0.$$
The sheaf $ \sE xt^1( \Omega^1_Y, \hol_Y)$ is, for a normal surface, concentrated on the singular points,
and the stalk at each of  these points is  equal to the quotient Tjurina algebra 
$$\hol_{\PP} / (F, \partial F /  \partial x_i),$$ where the $x_i$ are local (or projective) coordinates.

Unobstructedness means, when we pass to the long exact cohomology sequence,
in view of the surjection $ H^0(\hol_{\PP}(d)) \ra H^0(\hol_Y(d))$,  exactness of 
$$ 0 \ra H^0(N'_Y)  \ra H^0(\hol_Y(d))  \ra  H^0(\sE xt^1( \Omega^1_Y, \hol_Y)) \ra 0.$$

The main result is that then the projective space $\PP( H^0(\hol_{\PP}(d)))$, which parametrizes
surfaces of degree $d$, at the point corresponding to $Y$, has a submersion onto the manifold 
$\sL oc Def (Y)$ of local deformations of the singularities of $Y$ (see \cite{tjurina}). Hence in the neighbourhood of
$Y$ we can find all possible local deformations of the singularities of $Y$.

We have the following  criterion relating shortenings and partial smoothings:

\begin{theo}
\label{thm_d_realized}
Let $Y$ be an unobstructed surface of degree $d$ in $\PP^3$ with ADE singularities: then all 
 geometric driven shortenings of the code $\sK : = H_1^{\vee}$ are
realized by some normal surface of degree $d$, which is a
partial smoothing of $Y$. And, conversely, any partial smoothing of $Y$ yields 
a geometric driven shortening of the code $\sK : = H_1^{\vee}$.
\end{theo}

\begin{proof}
Let $\Sigma : = Sing (Y)$: it suffices to show the statement for a partial smoothing of one singular point $x$,
since we can independently achieve all the possible local deformations of the singularities. 

There is a local deformation $Y'$ of $Y$ which is equisingular at the singular points different from  $x$ 
 and achieves any local deformation of the singularity $x$.

By the Brieskorn-Tyurina theorem \cite{brieskorn} \cite{tjurina}, there is a smooth family with connected base
whose fibres contain  
the respective resolutions $\tilde{Y}$ and $\tilde{Y'}$,
and $Y^*$ is diffeomorphic to the complement of the union of all the exceptional curves $E_i$, while 
$(Y')^*$ is diffeomorphic to the complement of the union of some configurations of exceptional  curves $E'_j$,
 which includes the configurations 
of the curves $E_i$ which are exceptional and do not lie over $x$.

The local deformation amounts to a deformation of the tubular neighbourhood $T(D_x)$ of the exceptional divisor $D_x$,
and as proven by Burns and Wahl \cite{burns-wahl} Theorem 2.14 and Remark 2.15,  both  the deformations of $T(D_x)$
and the deformations of the germ $(Y,x)$ are smooth of the same dimension, 
and there is a map between them which is the quotient by the action of the Weyl group of
the singularity. The first ones amount to destroying some of the exceptional curves in $D_x$, 
and this  leads to a  geometric driven shortening of the code 	$H_1(x)$.

In fact, $H_1(Y^*,\ZZ) $ surjects onto $H_1((Y')^*,\ZZ) $ with kernel generated by the $\ga_i$ 's which are geometric loops around the irreducible curves $E_i$ which do not remain holomorphic submanifolds after the deformation. Then $H_1(x) \ra H _1(Y^*,\ZZ) $ , and $H_1(x)$ surjects onto $\oplus_yH_1(y)$,
where the $y$'s are the singularities of $Y'$ to which $x$ deforms.

The previous argument shows also the converse assertion.

\end{proof}

The next result follows essentially  from \cite{burns-wahl} and might be known to specialists; we give a simple 
and short  proof but without  giving full references for some well known facts.

\begin{theorem}
\label{singK3}
If $Y$ is a singular projective K3 surface, that is, $Y$ is a  surface with ADE singularities,
 and it has a trivial canonical divisor, then $Y$ is projectively unobstructed, that is, one can independently 
 smoothen all the exceptional curves 
  while preserving the ample  divisor class $H$.
 \end{theorem}
 
 \begin{proof}
 Indeed, as already observed, by \cite{burns-wahl} Theorem 2.14 and Remark 2.15, it suffices to show
 that the deformations of the smooth K3 surface $S$ and the normal crossing divisor $D = H + E = D_1 + \dots + D_m$,
 where $E$ consists of the ADE configurations, has a submersion onto the local deformations of the
 exceptional components of $D$.
 
 Now, infinitesimal deformations of the pair $(S,D)$, that is, preserving the components of $D$, are governed  by the cohomology group
 $$ H^1(\Theta_S (- log D_1, \dots, - log D_m)) \cong H^1 (\Omega^1_S (log D_1, \dots,  log D_m) (K_S))^{\vee},
 $$
 where the isomorphism is provided by Serre duality. These map to infinitesimal deformations of $S$, governed by 
 $ H^1(\Theta_S)$. 
 
 Use now the triviality of  $K_S$ to infer that $\Omega^1_S (log D_1, \dots,  log D_m) (K_S) \cong \Omega^1_S (log D_1, \dots,  log D_m)$, and the long exact cohomology sequence
 associated to the residue sequence 
 $$ 0 \ra \Omega^1_S \ra \Omega^1_S (log D_1, \dots,  log D_m) \ra \oplus_1^m \hol_{D_i} \ra 0.$$
 
 Since the intersection matrix $(D_i \cdot D_j)$ is non degenerate, the Chern classes of the curves $D_i$
 are linearly independent. And since, as observed in \cite{modulispace}, the coboundary map
 $\oplus_1^m H^0(\hol_{D_i}) \ra H^1 (\Omega^1_S)$ is given by the Chern classes of the divisors $D_i$, it follows that
this coboundary map is injective.

Since  $H^0 (\Omega^1_S)=0= H^2 (\Omega^1_S)$, and the $D_i$ are rational curves for $ i < m$,
 we have an exact sequence 
 $$ 0 \ra \oplus_1^m H^0 ( \hol_{D_i}) \ra H^1 (  \Omega^1_S)  \ra H^1 (\Omega^1_S (log D_1, \dots,  log D_m))
  \ra H^1 ( \hol_{D_m})  \ra 0.$$
  
  Dualizing it, one sees that, since $S$ is unobstructed, that is, $Def(S)$ is an open set in $H^1(\Theta_S)$
  (which is the dual of  $H^1 (  \Omega^1_S) $), the deformations of $S$ map onto the dual of $(\oplus_1^m H^0 ( \hol_{D_i}))$,
  so that one can independently smoothen all the exceptional curves $D_1 ,  \dots , D_{m-1}$
  while keeping the ample  divisor class $D_m$.
  
 \end{proof}

\section{ADE-Codes for surfaces $X$ with RDP=ADE singularities}

In the  case of nodal surfaces $X$ we get the usual binary code.

 If $\Sigma =  Sing (X) = \{P_1, \dots, P_{s}\}$, then
 the {\bf strict code} $\sK \subset (\ZZ/2)^{\Sigma}$ associated to $X$ is  the image of  the  injective
 linear map   dual to  the surjection
 $ (\ZZ/2)^{\Sigma} \twoheadrightarrow H_1(X^*, \ZZ)$. 
 
 One of the key dificulties of the theory, also in the nodal case,  is concerned with the determination of the possible sets of weights that a normal surface of degree $d$ in $\PP^3$ can have. The easiest restrictions come from applying the Riemann Roch theorem.
 
 We attempt to find   similar restrictions  for generalized ADE codes, recall  here 
 our notation $H_1 : = H_1(X^*, \ZZ)$.
 
 \subsection{Restrictions for the weights.} 
 Given a vector $ v \in H_1^{\vee}$ of order (exactly) $N$, it determines a surjection, which is the composition
 $$ \oplus_{x \in \Sigma}H_1 (x)  \twoheadrightarrow H_1 \twoheadrightarrow \ZZ/N.$$
 
 Now, $v$ determines a normal finite Galois covering $Y \ra X$ with Galois  group $\ZZ/N$, and branched only on
 a set $\Sigma(v)$ of singular points.
 
 \begin{remark}
 (1) Let the image of $H_1 (x)  \ra \ZZ/N$ be isomorphic to $\ZZ/N_x$, that is, equal to $(s_x \ZZ)/ (N \ZZ)$,
 where $s_x N_x = N$. Then 
 $$N= LCM (N_x)  for \ x \in \Sigma \  {\rm \ that \  is, \ }  1 = GCD(s_x),  for \ x \in \Sigma ).$$
 
 (2) For the case of ADE singularities $N_x \in \{ 2,3,4, s_x | n_x+1\}$, where $N_x > 4$
 only in the case where $x$ is of type $A_{n_x}$ with $n_x \geq 4$.
 \end{remark}
 
 We assume now that $X$ has only ADE singularities, and let $\Sigma(v)$ be the subset of singular points of $X$
 for which $H_1 (x)  \ra \ZZ/N$ is nontrivial, that is, $\Sigma(v)$ is the set of branching points for $f : Y \ra X$.
 
 Clearly the $E_8$ points are not in $\Sigma(v)$. For each $x$,  let $N_x$ be the integer such that
 the image of  $H_1 (x)  \ra \ZZ/N$ is isomorphic to  $\ZZ/N_x$.
 
Then the singular set  of $Y$ is contained in the inverse image of $\Sing(Y) = \Sigma$, and
inverse image of $x$ consists of $N/N_x$ points which are (possibly) singular points, corresponding to the surjection
  $H_1 (x)  \ra \ZZ/N_x$.
  
  Set for simplicity of notation $ r : = N_x$. We shall now determine which type of ADE singularity are
  the points lying over $x$.
  
  Observe preliminarly that, if  $H_1 (x) $ is cyclic, then the integer $r$ determines a unique type of singularity, whereas
  for  $H_1 (x) = (\ZZ/2)^2$, we have three possibilities (and this only occurs for $x$ of type $D_{2m}$).

  \begin{enumerate}
  \item
  For $A_n$, we have then $(n+1) = r s$ and we get $r$ singular points of type $A_{s-1}$, since the
  local  covering is
  $ \CC^2 / \mu_s \ra \CC^2 / \mu_{n+1}$, $\mu_s$ being the cyclic group of $s$-th roots of $1$.
  \item
  For $D_{2m+1}$,  $H_1 (x) $ is cyclic of order $4$. 
  
  Consider  first the case with  $r=4$: here we have a surjection
  $ \hat{D}_{2m-1} \ra \ZZ/4$ whose kernel is the cyclic group $\mu_{2m-1}$. In fact   $\hat{D}_n$ is generated
  by $\s (u,v) = (i v, iu)$ and by $\rho (u,v) = (\zeta_n u, \zeta_n^{-1} v)$, and moreover $\rho$ is in the  commutator
  subgroup  for $n$ odd, because 
 $\s \rho \s^{-1} = \rho^{-1}\Rightarrow [\s, \rho] = \rho^{-2}.$

  The conclusion is that,  over $x$, lie $N/4$ singularities of type $A_{2m-2}$.
  \item
 For $D_{2m+1}$,  $r=2$,  we have a surjection
  $ \hat{D}_{2m-1} \ra \ZZ/2$ and the kernel is $\cong \ZZ/2 \times \mu_{2m-1}$, hence we find 
  $N/2$ singularities of type $A_{4m-3}$.
  \item
  For $D_{2m+2}$,  $r=2$,  we have a surjection
  $ \hat{D}_{2m} \ra \ZZ/2$, factoring through the abelianization $H_1(x) \cong (\ZZ/2)^2$.
     $ \hat{D}_{2m} $ is generated by $\s$ as above, and $\rho (u,v) : = (\zeta_{4m}^{-1} u, \zeta_{4m} v)$.
     In the abelianization $ 2 [\s] = 2 [\rho]=0$. We have three cases, according to the choice 
     of  an index 2 subgroup as the kernel.
     \item
     Taking as respective kernels:
     $\langle \rho \rangle \cong \mu_{4m}$ ,  $\langle \s , \rho^2 \rangle \cong  \hat{D}_m$,  $\langle \s \rho , \rho^2 \rangle \cong  \hat{D}_m$,
  we have accordingly $N/2$ singularities of respective types $A_{4m-1}, D_{m+2}, D_{m+2}.$
     \item
     For $E_7$, the double covering corresponds to the inclusion $\frak A_4 < \frak S_4$,
     and we get  $N/2$ singularities of type $E_6$.
     \item
     For $E_6$,  $H_1(x) = \ZZ/3$, hence $r=3$, which  corresponds to the surjection $\frak A_4 \ra \frak A_3$,
     with kernel $ (\ZZ/2)^2$, hence we get  $N/2$ singularities of type $D_4$.
  \end{enumerate}
  
  \begin{remark}
  i) In \cite{ardp} Theorem 2.1, page 89 , table 2 page 90,  it is shown  that $A_{2k+1}/b \cong D_{k+3}$,
  where  $A_{2k+1} = \{ z^2 + x^2 + y^{2k+2}=0\}$, and $b (x,y,z) : = (x, -y,-z)$.
  
  While, for $a (x,y,z) : = (x, -y,z)$, $A_{2k+1}/a \cong A_k$, but this quotient is ramified in codimension 1,
  while the quotient  for $e (x,y,z) : = (-x, -y,-z)$, $A_{2k+1}/e $, is not an ADE singularity (see Theorem 2.5,
  page 92 of \cite{ardp}).
  
  ii) It is fun to make a picture of the triple covering in the case $E_6$ for the minimal resolution. Algebraically,
  we get equations from $ z^2 = x^3 + y^4$ and $w^3 = x/y$.
  
  iii) Taking the full covering associated to $H_1(x)$ for the case $D_{2m}$, we get $A_{2m-3}$ singularities, see also table 3,
  page 93 of \cite{ardp}.
  \end{remark}
  
    \begin{remark}\label{primary}
  In view of the primary decomposition (Proposition \ref{primary}), 
  in order to obtain restrictions, it suffices to consider the case where we have a vector $v$ of order which
  is a prime power $N = p^k$.
 And recall  that, for $ p>3$,  we get only a contribution from $A_n$, with $ p | (n+1)$,
 for $p=3$ only from $A_{3k-1}$ and $E_6$, for $p=2$ from $A_{2m+ 1}$ , $D_n$,  $E_7$.
  \end{remark}
  
\subsection{Restriction on the weights for vectors $v$ of order $N=2$.}Let $S$ be, unlike elsewhere,  the minimal resolution of the singular points  
of $X$, $x \in \sN \subset \Sigma$,  
contributing to the vector $v$.

Then the double covering $f : Y \ra X$ induces a double covering $ F : Z \ra S$, determined by a line bundle $L$ on $S$
such that
$$ 2 L \equiv \sum_{x \in \sN} B_x,$$
where $B_x$ is a reduced exceptional divisor, such that $B_x \neq 0 \ for \ x \in \sN$.

We need first to calculate the self intersection numbers $B_x^2$. For this purpose we need to observe that, if $H_1(x)$
is cyclic, and $x \in \sN$, then $B_x$ consists of the curves for which $\ga_i$ is an odd multiple of the generator.

\begin{enumerate}
\item
$A_n$ appears in $\sN$ only if $n = 2m+1$ is odd, in this case  $H_1(x) \cong \ZZ/(2m+2)$ and $B_x$ consists of the
odd labelled curves, these are $m+1$ disjoint $(-2)$-curves, hence $$B_x^2 = -2(m+1).$$
\item
For $E_7$ we get, using our previous calculations, $B_x = D_1 + D_3 + D_7$, hence $$B_x^2 = -6.$$
\item
For $D_n$, $n$ odd, $B_x = C_n + C_{n-2}$, hence $$B_x^2 = -4.$$
\item
For $D_{2m}$, we have generators $\ga_1, \ga_n$, and three options for the surjection to $\ZZ/2$.
But the two options: $\ga_{2m} \mapsto 0$,  respectively $\ga_{2m-2} \mapsto 0$, are related by a symmetry.

In the first  case $B_x = C_1+ C_3 + \dots + C_{2m-3} + C_{2m-1} $, hence $$B_x^2 = -2 m.$$

If instead $\ga_1 \mapsto 0$, then $B_x = C_{2m} + C_{2m-1} $,  hence $$B_x^2 = -4.$$

\end{enumerate}

\begin{defin}
In the case of a singularity of type $D_{2m}$, the previous computations suggest to distinguish the possible
surjections $H_1(x) \ra \ZZ/2$ as being of subtype $(-)$ if  $\ga_1 \mapsto 0$, and of subtype $(+)$
otherwise. Because, for subtype $(-)$,  $B_x^2 = -4$, while, for subtype $(+)$, $B_x^2 = -2 m$.

\end{defin}

\begin{prop}\label{N=2}
Let $v$ be a code vector of order $N=2$. Define $\sN$ as the subset of the singular points $x \in \Sigma$
such that $H_1(x)$ surjects onto $\ZZ/2$. Let $t(A_{2m+1})$ be the $A_{2m+1}$-weight of $v$,
that is, the number of points in $\sN$ which are of type $A_{2m+1}$.
Let similarly $t(E_7)$ be  the number of points in $\sN$ which are of type $E_7$,
$t(D_n, -)$ the number of points in $\sN$ which are of type $D_n$, and of subtype  $(-)$ for $v$ if $n$ is even,
$t(D_{2m}, +)$ the number of points in $\sN$ which are of type $D_{2m}$, and of subtype  $(+)$ for $v$.

Then $$ (**) \   \sum_m ((m+1)\   t(A_{2m+1}) + m \  t(D_{2m}, +)) + \sum_n 2 \  t(D_n, -)  + 3 \ t(E_7)$$
is divisible by $4$, and indeed by $8$ if $K_S$  is divisible by 2.
\end{prop}
\begin{proof}
Observe that $F_* (\hol_Z) = \hol_S \oplus \hol_S (-L)$, and by Riemann Roch 
$$\chi(-L) = \chi(\hol_S ) + \frac{1}{2} L (L+ K_S)  = \chi(\hol_S ) + \frac{1}{2} L^2 \Rightarrow   2 | L^2.$$
In fact  $K_S L =0$, in view of 
$ 2 L \equiv \sum_{x \in \sN} B_x,$ and since the divisors $B_x$ are exceptional and $K_S$ is the pull back of $K_X$.  The same formula implies 
$$\sum_{x \in \sN} B_x^2 \equiv 0 \ (mod \ 8),$$ equivalent to $(**)$.

To get divisibility of $(**)$ by $8$, we need that $L^2 $ be divisible by $4$, that is, that
$$\chi(\hol_Z ) = 2  \chi(\hol_S ) + \frac{1}{2} L^2 $$
be divisible by $2$. 

This is true, as in Proposition 2.11 of \cite{babbage}, because:  if $Y \ra X$ is the corresponding 2-1 covering of normal surfaces,
then $K_Y$ is divisible by $2$,
being the pull back of $K_X$, hence there exists a Cartier divisor $G$ such that $K_Y \equiv 2 G$.

Then   $\chi (\hol_Y(G) ) $ is an even integer  since it equals $ 2 h^0(\hol_Y(G) ) + h^1 (\hol_Y(G) ) $,
where the last vector space $H^1 (\hol_Y(G) )$  carries a bilinear alternating non degenerate form, hence has even dimension.
Now, $ \chi(\hol_Z ) = \chi(\hol_Y)$  and  
 $\chi (\hol_Y(G) ) = \chi(\hol_Y ) -\frac{1}{2} G^2$ hence  we are done if $G^2$ is divisible by $4$.
 
 We conclude because  $G = f^*(M), $ hence  $G^2 = 2 M^2$ and $M^2$ is even because the intersection form is even.

\end{proof}

\subsection{Restriction on the weights for vectors $v$ of order $N=3$ or a larger  number.}Let $S$ be
now  the minimal resolution of the singularities 
$x \in \sN \subset \Sigma$ of $X$
contributing to the vector $v$.

Then the triple covering $f : Y \ra X$ induces a singular triple covering $ F : Z \ra S$, determined by a line bundle $L$ on $S$
such that
$$ 3 L \equiv \sum_{x \in \sN} B_x,$$
where $B_x$ is an exceptional divisor, $B_x \neq 0 \ for \ x \in \sN$; moreover the multiplicities of the components $D_i$ 
in $B_x$ are either $1$ or $-1$ according to the image of $\ga_i$ in $\ZZ/3$ being $1$ or $2$.

\begin{enumerate}
\item
$A_n$ appears in $\sN$ only if $n+1  = 3s$, in this case  $H_1(x) \cong \ZZ/(3s)$ and $B_x$ consists of the
 curves with label congruent to $1$ modulo $3$,  minus the curves with label congruent to $2$ modulo $3$.
 
 The associated reduced divisors  form $s$ $A_2$ strings,  hence $$B_x^2 = - 6 s.$$
\item
For $E_6$ we get, using our previous calculations, $B_x = D_1 - D_2 + D_5 - D_6$, 
 hence $$B_x^2 = - 12.$$
\end{enumerate}

\begin{prop}\label{N=3}
Let $v$ be a code vector of order $N=3$. Define $\sN$ as the subset of the singular points $x \in \Sigma$
such that $H_1(x)$ surjects onto $\ZZ/3$. Let $t(A_{3s -1})$ be the $A_{3s -1}$-weight of $v$,
that is, the number of points in $\sN$ which are of type $A_{3s -1}$.
Let similarly $t(E_6)$ be  the number of points in $\sN$ which are of type $E_6$,

Then $$ (***) \   \sum_s s \  t(A_{3s -1}) + 2  t(E_6) \equiv  \sum_s s \  t(A_{3s -1}) -  t(E_6)$$
is divisible by $3$.
\end{prop}
\begin{proof}
Observe that $F_* (\hol_Z) = \hol_S \oplus \hol_S (-L) \oplus \hol_S (-L)$. 

And, as in the proof of Proposition \ref{N=2},  by Riemann Roch 
$$\chi(-L) = \chi(\hol_S ) + \frac{1}{2} L (L+ K_S)  = \chi(\hol_S ) + \frac{1}{2} L^2 \Rightarrow   2 | L^2.$$

Since 
$ 3 L \equiv \sum_{x \in \sN} B_x,$ we get 
$$ 9 L^2 = (\sum_{x \in \sN} B_x)^2 = - 6 [ 2 (t(E_6) +  \sum_s s \  t(A_{3s -1}) ].$$

\end{proof}

\begin{cor}
Assume that the singular points of $Y$ are just $A_2$ singularities, cusps in the terminology
of \cite{b-r-1},  \cite{b-r-2}. 

Then we have a vector code $\sK$ over $\FF_3$, and the weight
of a code vector is always divisible by $3$. 

\end{cor} 

Barth and Rams proved that for sextic surfaces in $\PP^3$ with cusps, the minimal weight is $18$, and that the weights
$18,24, 27$ can occur.

Barth showed that, for quartics, there cannot be $ 9$ cusps, but there is a quartic with $8$ cusps (\cite{b-9}, \cite{b-8}).

\begin{defin}
A code vector $v$ of order $N$ is said to be  {\bf almost simple} if $(w(v; j, \xi)=0$ for $V_j$ non cyclic, 
and, for  $V_j$  cyclic, for $\xi(1) \neq 1, N-1$.

\end{defin}

\begin{prop}\label{N>3}
Let $v$ be an almost simple  code vector of order  $N > 4$. Denote $\sN$  the subset of the singular points $x \in \Sigma$
such that $H_1(x)$ surjects onto $\ZZ/N$: these are $A_n$ strings, where  $n+1 = N s$.

 Let $t(A_{Ns -1})$ be the $A_{Ns -1}$-weight of $v$,
that is, the number of points in $\sN$ which are of type $A_{Ns -1}$ (under our assumption  $t(A_{Ns -1})  = w (v; A_{Ns-1}, 1) + w (v; A_{Ns-1}, N-1)$).

Then $$ (****) \  \sum_s   s \ t(A_{Ns -1})$$
is divisible by $N$.
\end{prop}
\begin{proof}
Again  $$F_* (\hol_Z) = \hol_S \oplus \hol_S (-L) \oplus \hol_S (- 2 L) \oplus \dots \oplus \hol_S (- (N-1)L).$$ 

And, as in the proof of Proposition \ref{N=2},  by Riemann Roch $2|L^2$.

Now, for $x \in \sN$, $B_x$ consists of $s$ strings of $A_{N-1}$ singularities, with multiplicities $1,2 \dots, N-1$.
We can write such a string as $Z_{N-1}  + Z_{N-2} + \dots + Z_1$, where the $Z_j$'s are the fundamental cycles
of the string where the first $N-1-j$ curves are deleted.
Then by induction $$(Z_{N-1}  + Z_{N-2} + \dots + Z_1)^2  = -2 - 2(N-2) + (Z_{N-2} + \dots + Z_1)^2= $$
$$ = -2(N-1) + (Z_{N-2} + \dots + Z_1)^2 = - N (N-1).$$
It follows that for each string we get a contribution to $B_x^2 $ equal to $- N (N-1).$

Since 
$ N L \equiv \sum_{x \in \sN} B_x,$ we get 
$$ N^2 L^2 = (\sum_{x \in \sN} B_x)^2 = - N (N-1) \ \sum_s s \  \
 t(A_{Ns -1}),$$
and $(****)$ follows.

\end{proof}

In general, for $N \geq 4$, one has more complicated formulae
depending on the refined weights. We treat for instance the case $N=5$.

\begin{prop}\label{N=5}
For vectors $v$ of order $N=5$, we get that 
$$ 5 \  |  \ [ \sum_s w (v; A_{5s-1}, 1) + w (v; A_{5s-1}, 4) - w (v; A_{5s-1}, 2) - w (v; A_{5s-1}, 3)].$$ 
\end{prop}
\begin{proof}
The points $x \in \sN$ with $v_x(1) = 1, 4,$ give the same contribution as the one calculated in Proposition \ref{N>3},
namely $-20$ while those with $v_x(1) = 2,3,$ since the multiplicities of the curves in $B_x$ are $(2,4,1,3)$
or these in reverse order, give a contribution equal to  $ -2 (30) + 2 (8 + 4 + 3) = -30$; dividing by $5$ we get 
coefficients $+1$, respectively $-1$,
modulo $5$.

\end{proof}
\bigskip

\section{The binary code  $\sK [2]$ of  surfaces in $\PP^3$ with ADE singularities, and the B-inequality}

For  $Y$   a normal surface in $\PP^3$, with ADE singularities,   let $\tilde{Y} $ be the minimal resolution
of  the singular points.

We shall denote by $\sE$ the set of irreducible exceptional curves $E_i$ in $\tilde{Y} $: these
are all smooth rational curves with selfintersection $-2$.

\begin{defin}
Given a generalized ADE code 
$$ \sK : = Im (\ks^{\vee}), \ \ks^{\vee}: H_1^{\vee}  \hookrightarrow \ \oplus_{x \in \Sigma} H_1(x)^{\vee} ,$$
we define  $\sK[2]$  to be the set of code vectors of order $N=2$.  

$\sK[2]$  is a $\ZZ/2 = : \FF_2$ vector space, and  indeed we have
$$ \sK[2] = Hom (\ZZ/2, H_1^{\vee}) \cong  Hom ( H_1, \ZZ/2)=  Hom (H_1 ( Y^*, \ZZ), \ZZ/2 ) = H^1 ( Y^*, \ZZ/2).$$

\end{defin}

\begin{cor}\label{binarycode}
Let $Y$  be an ADE surface in $\PP^3$, and let $\tilde{Y} $ be the minimal resolution of  its singularities: 
then, setting $\nu : = | \sE|$,  the  binary code of $Y$, $\sK[2] \subset (\ZZ/2)^{\nu}$,
equals to the kernel of the map 
$$  \oplus_1^{\nu} (\ZZ/2 ) [E_i] \ra H^2 ( \tilde{Y}  , \ZZ/2).$$
The vectors  $v \in \sK[2]$ correspond to subsets $\sE_v$ of $\sE$ such that $\sum_{i \in \sE_v} E_i$ is linearly equivalent
to $2L$, for some divisor class $L$ on $\tilde{Y}$.

  The cardinality $t$ of such sets $\sE_v$, i.e., the weights of the code vectors in $\sK[2]$, 
are divisible by $4$, and indeed by $8$
if the degree $ d : = deg (Y)$ is even.

Finally, the dimension $k[2]$ of $\sK[2]$ satisfies

$$ (B) \ k[2] \geq \sum_x \de_x  - [ \frac{1}{2} \  b_2(\tilde{Y})].$$
Here,   $[[a]]$ denoting   the smallest integer bigger than $a$,
$\de_x  : = [[ \frac{n}{2}]] $ for $A_n$-singularities, $\de_x  : = [[ \frac{n+1}{2}]] $ for $D_n$-singularities,
$\de_x  : = 3,4,4$ for the respective singularities $E_6, E_7, E_8$.
\end{cor}
 
\begin{proof}
Since $Y$ has ADE singularities, by the theorem of Brieskorn and Tjurina  \cite{brieskorn} \cite{tjurina}, $\tilde{Y} $ is diffeomorphic to a smooth surface of degree $d = deg(Y)$
in $\PP^3$, which is simply connected, by Lefschetz' theorem.  In particular, $H^2 (\tilde{Y}, \ZZ)$ is a free $\ZZ$-module,
  $H^2 (\tilde{Y}, \ZZ/2 ) = H^2 (\tilde{Y}, \ZZ) \otimes_{\ZZ} \ZZ /2$,
and the second assertion follows immediately from the first.

By theorem \ref{jdg}, applied to $ X : =  \tilde{Y} $ with $D : = E_1 \cup \dots \cup E_{\nu}$,
we get that $$H_1 (Y^* , \ZZ) \cong coker [ H^2 (\tilde{Y}, \ZZ) \ra  \oplus_1^{\nu} \ZZ [E_i] ].$$


By the  exact sequence 
$$H^2 (\tilde{Y}, \ZZ) \ra  \oplus_1^{\nu} \ZZ [E_i] \ra H_1 (Y^* , \ZZ) \ra 0,$$
we tensor with $\ZZ/2$ and obtain an exact sequence
$$H^2 (\tilde{Y}, \ZZ/2) \ra  \oplus_1^{\nu} \ZZ /2 [E_i] \ra H_1 (Y^* , \ZZ/2) \ra 0,$$
and taking the dual $\ZZ/2$-vector spaces, we obtain the exact sequence  
$$ 0 \ra \sK[2] \hookrightarrow  \oplus_1^{\nu} (\ZZ/2 ) [E_i] \ra H^2 ( \tilde{Y}  , \ZZ/2).$$

The statement about $| \sE_v|$ was proven in \cite{babbage}, Proposition 2.11,
and reproven here in Proposition \ref{N=2}.
The final assertion follows since the dimension of an isotropic subspace for the intersection form, 
which is nondegenerate,  is at most 1/2 of the dimension $b_2(\tilde{Y})$ of $H^2 ( \tilde{Y}  , \ZZ/2)$.

Then we  observe that for each singularity $x$ we have $\de_x$ equal to the maximal number of disjoint
$(-2)$-curves, which generate an isotropic subspace.

\end{proof}

\begin{remark}

i) The B-inequality is named after Beauville, see \cite{angers}.

ii) If $\tilde{Y}$  is a K3 surface, $b_2(\tilde{Y}) = 22$, hence $\sK[2]$ has dimension at least $\sum_x \de_x - 11$.

iii) The generalized (B)-inequality was observed and used effectively in \cite{peters}.
\end{remark}

\section{The extended code $\sK'$ and sublattice saturation.}

In this section we extend the notion of the extended code $\sK'$ from the case of nodal surfaces
to the case of normal surfaces, in particular, surfaces with ADE singularities.

Recall that, for   $Y$  a nodal surface in $\PP^3$ of even degree $d$, if  $H$ is a smooth hyperplane section of $Y$,
then the {\bf extended code} $\sK '$ is defined as the kernel 
$$\sK ': =  ker [  \oplus_1^{\nu} (\ZZ/2 ) [E_i] \oplus (\ZZ/2 )  [H]  \ra H^2 ( S , \ZZ/2)].$$

More generally, if  $Y$ is a normal projective surface, and  $S = \tilde{Y} $ is  the minimal normal crossing resolution of singularities,
we set $D : = E_1 \cup \dots \cup E_{\nu} \cup H$, where the $E_i$'s are the exceptional curves, and $H$ corresponds to the inverse image of a smooth hyperplane section,
disjoint from $\Sigma : = Sing(Y)$.

We apply now   Theorem \ref{jdg} to the case where 
$S$ is  a smooth compact surface with $H_1( S  , \ZZ)=0$,  $D = D_1 \cup \dots \cup D_{\nu +1}$ 
is a normal crossing divisor, with $D_i$ is irreducible, hence 
 $ H_1( S \setminus D , \ZZ)$ is the cokernel of
$$ \rho : H^{2} (S, \ZZ) \ra  H^{2} (D, \ZZ) = \oplus_1^{\nu +1} \ZZ [D_i],$$ 
and $\rho(L) = \sum_i (L \cdot D_i) [D_i]$.

\begin{defin}
The underlying group of the extended code $\sK'$ is defined  as the dual of $ H_1 (Y^* \setminus H , \ZZ).$

Its definition  as a generalized labeled code is through the  surjection \ref{Extended-code} of 
the   next Theorem \ref{codehomology2}.

\end{defin}

\begin{theorem}\label{codehomology2}
Let $Y$  be a normal surface in $\PP^3$,  let $S $ be a minimal normal crossing resolution of $Y$, and let $S'$ be the complement
in   $S$ of the union of the  exceptional curves $E_i$ with  a smooth hyperplane section $H$ (not passing through any node).

(i) Then the first homology group $H_1( S' , \ZZ)$ is a quotient 
$$  \oplus_{x \in \Sigma}H_1(x)  \oplus (\ZZ/d ) [H] \twoheadrightarrow  H_1 (S' , \ZZ),$$ 
where $d= H^2$ is the degree of $Y$;

(ii) moreover,   its extended  code $\sK'$, which 
has underlying group  equal to  the dual of $H_1 (S' , \ZZ) = H_1 (Y^* \setminus H , \ZZ) $
determines the saturation  of the lattice $\sL$ generated by the
curves $E_i$ and by the hyperplane section $H$ inside the  lattice $
\Lambda : = H^2 (S, \ZZ)  $.

 More precisely, the saturation $\sL^{sat}$ is such that   $\sL^{sat} / \sL \cong (\sK')^{\vee}$,
 and we have a surjection
 \begin{equation}\label{Extended-code}  \oplus_{x \in \Sigma}H_1(x)  \oplus (\ZZ/d' ) [H] \ra (\sK')^{\vee},
 \end{equation}

 where $d' : = GCD (d,m)$, $m : = LCM (\{ m_x, x \in \Sigma\}$, and $m_x$ is the exponent of the group $H_1(x)$.

 The same result holds for normal surfaces $Y$ whose minimal resolution $S = \tilde{Y} $ has $H_1(S, \ZZ)=0$ 
and is such that the hyperplane class $H$ yields an indivisible element in $Pic(S)$.
\end{theorem}

\begin{proof}

We have  the exact sequence 
$$ \rho : \Lambda : = H^{2} (S, \ZZ) \ra  H^{2} (D, \ZZ) = \oplus_1^{\nu} \ZZ [E_i] \oplus \ZZ H  = : \sL \ra H_1( S \setminus D , \ZZ) \ra 0 ,$$ 
and $\sL$ embeds in $\Lambda$ because the intersection product on $ \sL$ is non degenerate.

Then $\sL / \rho  (\Lambda) $ is a quotient  of $\sL / \rho (\sL) =  \oplus_{x \in \Sigma}H_1(x)  \oplus (\ZZ/d ) [H]$
and (i) follows.

We defined  the group underlying the code $\sK'$ as  the dual of
  \begin{equation}\label{Extended} 
  H_1 (Y^* \setminus H , \ZZ) \cong coker ( H^2 (S, \ZZ) \ra  \oplus_1^{\nu} \ZZ [E_i] \oplus \ZZ H ),
  \end{equation}
   
the cokernel  of the dual map of the embedding inside the  lattice $
\Lambda : = ( H^2 (S, \ZZ) = H^2 (\tilde{Y}, \ZZ)  $ of the lattice $\sL$ ( generated by the
curves $E_i$ and by the hyperplane section $H$).

In view of the exact sequence 
$$0  \ra \sL \ra \Lambda \ra F \oplus T \ra 0 ,$$
where $T$ is a torsion group and $F$ a free abelian group, dualizing we get (in view of the unimodularity of $\Lambda$)
$$0  \ra F^{\vee} \ra \Lambda \ra \sL^{\vee} \ra Ext^1( T, \ZZ) \ra 0,$$
showing that $T \cong H_1 (Y^* \setminus H , \ZZ)$.

 Consider now the saturation $\sL' : =  \sL^{sat} $ of $\sL$ inside $\Lambda$; $\sL' $  is primitively embedded.

Since $ \sL \subset \sL'$, and $ \sL' /  \sL \cong T$, it follows that $\sL'  =  \sL^{sat} $ is generated by $\sL$
and by classes $\la \in \Lambda$ of the form 
$$  \la =  \sum_{x \in \Sigma} (\frac{1}{m_x} \sum_i a_i E_{i,x} )+ \frac{1}{d} b H , a_i, b \in \ZZ,$$
where $m_x$ is the exponent of the group $H_1(x)$.

Since $m : = LCM \{ m_x, x \in \Sigma \}$, then, 
since $m \la \in \Lambda$, we obtain that $ \frac{mb}{d}  H \in \Lambda$.

But $H$ is indivisible (cf. for instance \cite{babbage}, page 464), hence $d : = H^2$ divides $mb$; if $d$ is relatively prime to $m$, this implies that $d$ divides $b$,
hence we may assume $b=0$. Else, setting $d' = GCD (m,d)$,  $d = d' c, m = d' c'$, then $b = c b'$, hence
$  \la =  \sum_{x \in \Sigma} (\frac{1}{m_x} \sum_i a_i E_{i,x} )+ \frac{1}{d'} b' H , a_i, b' \in \ZZ.$

 The  only facts that we have used
 are that $Pic(\tilde{Y})$ is a free abelian group, because $H_1(S)=0$, 
and that the class of $H$ is indivisible.

\end{proof}

\subsection{Weight restrictions for the extended code} 
If we have now an extended code
$$   \sK' \hookrightarrow   \oplus_{x \in \Sigma}H_1(x)^{\vee}   \oplus (\ZZ/d' ) [H] $$

and a vector $v$ which is not in $\sK$, hence with `second' coordinate $m \in \ZZ/d' $,
we can again consider the primary decomposition, and obtain restrictions for the labeled weight
using the same method as for $\sK$, namely taking the cyclic covering branched on $H$ and
on a subset $\sN$ of  the singular set $\Sigma$.

 We shall work out more systematically these restrictions in the sequel to this first part, but let us give easy extensions of
 the previous Propositions \ref{N=2} and \ref{N=3}.

\begin{prop}\label{extended-N=2}
Let $v$ be a code vector in $\sK' \setminus \sK$ of order $N=2$. Define $\sN$ as the subset of the singular points $x \in \Sigma$
such that $H_1(x)$ surjects onto $\ZZ/2$. Let $t(A_{2m+1})$ be the $A_{2m+1}$-weight of $v$,
that is, the number of points in $\sN$ which are of type $A_{2m+1}$.
Let similarly $t(E_7)$ be  the number of points in $\sN$ which are of type $E_7$,
$t(D_n, -)$ the number of points in $\sN$ which are of type $D_n$, and of subtype  $(-)$  for $v$ if $n$ is even,
$t(D_{2m}, +)$ the number of points in $\sN$ which are of type $D_{2m}$, and of subtype  $(+)$  for $v$.

Then $$ (**) \   \sum_m ((m+1) \  t(A_{2m+1}) + m  \ t(D_{2m}, +)) + \sum_n 2 t(D_n, -)  + 3 t(E_7) - \frac{d}{2} - K_S \cdot H$$
is divisible by $4$.
\end{prop}
\begin{proof}
Observe again that $F_* (\hol_Z) = \hol_S \oplus \hol_S (-L)$, and by Riemann Roch 
$$\chi(-L) = \chi(\hol_S ) + \frac{1}{2} L (L+ K_S)  = \chi(\hol_S ) + \frac{1}{2} (L^2 + L K_S).$$
We have  $K_S L = \frac{1}{2} K_S  H $, in view of 
$ 2 L \equiv \sum_{x \in \sN} B_x + H,$ and since the $B_x$ are exceptional. 

Hence  $8$ divides 
$$ 4 L^2 +  2 K_S  H =  \sum_{x \in \sN} B_x^2  + H^2 + 2 K_S  H,$$ 
 and this condition is equivalent to $(**)$.

\end{proof}

\begin{prop}\label{extended-N=3}
Let $v$ be a code vector in $\sK' \setminus \sK$  of order $N=3$, and let $v(H)= : \e = \pm 1$. 

Define $\sN$ as the subset of the singular points $x \in \Sigma$
such that $H_1(x)$ surjects onto $\ZZ/3$. Let $t(A_{3s -1})$ be the $A_{3s -1}$-weight of $v$,
that is, the number of points in $\sN$ which are of type $A_{3s -1}$.
Let similarly $t(E_6)$ be  the number of points in $\sN$ which are of type $E_6$,

Then $d$ is divisible by $3$ and 
$$ (***) \   \sum_s s \  t(A_{3s -1}) -  t(E_6) + \frac{d}{3} +   \e K_S H$$
is divisible by $3$.
\end{prop}
\begin{proof}
Observe that $F_* (\hol_Z) = \hol_S \oplus \hol_S (-L) \oplus \hol_S (-L)$. 

And, as in the proof of Proposition \ref{N=2},  by Riemann Roch 
$$\chi(-L) = \chi(\hol_S ) + \frac{1}{2} L (L+ K_S) .$$

Since 
$ 3 L \equiv \sum_{x \in \sN} B_x + \e H,$ we get 
$$ L (L+ K_S)  = L^2 + \frac{1}{3} \e K_S H,$$ and 
$$ 9 L^2 = (\sum_{x \in \sN} B_x)^2  + H^2 , \ \  9 L (L+ K_S) = (\sum_{x \in \sN} B_x)^2  + H^2 + 3 \e K_S H,$$
hence  $18$  divides 
$$= - 6 [ 2 (t(E_6) +  \sum_s s \ t(A_{3s -1}) ] + d + 3  \e K_S H.$$

 The previous formulae show that $d$ is divisible by $3$,
and that $d + 3 \e K_S H$ is an even number, hence
$3$ divides 
$$= - 2 [ (2 t(E_6) +  \sum_s   s \ t(A_{3s -1}) ] + \frac{d}{3}  +   \e K_S H,$$
and this is equivalent to
$$ (***) \   \sum_s s \  t(A_{3s -1}) -  t(E_6) + \frac{d}{3}+   \e K_S H $$
being  divisible by $3$.

\end{proof}

\section{Examples: the ancestors for  surfaces of degree $\leq 3$ and some K3 ancestors}

If $Y$ is a singular normal surface in $\PP^3$ of degree $d$, then,  for $d=2$,  $Y$ must be
 the quadric cone. 

Moreover, for all $d$, $Y$ has ADE singularities \cite{artin} if the singular 
points are rational and have multiplicity 2.

Even if  almost everything is known for cubic surfaces ($d=3$), we want to describe the ancestors 
from our point of view (see especially Table 9.1 of \cite{cag}, and also Section 8 of \cite{cat-kummer}).

Observe before hand that  the minimal resolution $S$ of $Y$ has $b_2(S)=7$, hence  the number of exceptional curves 
is smaller or equal to  6, and if equality holds, we surely have an ancestor (since this number, 
which we called rank, is maximal).

\begin{enumerate}
\item
For nodal surfaces the ancestor is the $4$-nodal Cayley cubic $$\sC_3 :=  \{ x | \sum_0^3 \frac{x_0 x_1 x_2 x_3}{x_j} = 0\}.$$
It is also an ADE ancestor because $\sC_3$ is the only cubic with at least four singular points.
\item
 The next examples are  ancestors because the rank, i.e. the number of exceptional curves,  will be equal to  6.
 \item
Another ancestor is the cubic with 3 $A_2$ singularities (3 cusps)  $$ \sY^{cu}_3 : = \{ x | x_0^3 -  x_1 x_2 x_3= 0\}.$$
\item
A third  ancestor is the cubic with an $E_6$-singularity,  $$ \sY^{E_6}_3 : = \{ x | x_0 x_3^2 -   g(x_1, x_2, x_3)= 0\},$$
where, at the point $(1,0,0,0)$, the plane cubic $g =0$ has the line $x_3=0$ as an inflectionary tangent.
\item
The fourth ancestor is   the surface  $S$ obtained blowing up  6 infinitely near  points in the plane such that the first 3 do not lie
on a line  but  the 6 do  lie on  a conic: then the anticanonical system of $S$ maps to 
a cubic $Y^{5,1}$ which 
 has an $A_5$ singularity and a disjoint $A_1$ singularity.
 
  Because,  if $E'_i$ is the total transform of the i-th point, then  $H : =  3 L - \sum _i E'_i$ 
 has $H^2=3$ and contracts each $E_i, \forall  i \geq 2$, and also contracts  the proper transform
 of the conic $ Q \in | 2 L -  \sum _i E'_i|$.
\end{enumerate}

We can now describe the associated codes, even without using  the restrictions for the weights of 
$\sK'$.

But we observe that Propositions \ref{N=3} and \ref{extended-N=3} imply that for vectors in $\sK$
(of order $3$) the difference between the number of $A_2$-singularities and the number of $E_6$-singularities 
appearing is divisible by 3. Whereas, for vectors in 
$\sK'  \setminus \sK$ the same number is congruent to $0$ or $2$ modulo 3.

 In general the sublattice $\sL = \sL_0 \oplus \ZZ H$, where $\sL_0 : = \oplus_i \ZZ E_i$,
 is  contained in the lattice $\Lambda = H^2(S, \ZZ)$, and we have the inclusions
 $$ \sL = \sL_0 \oplus \ZZ H \subset  \sL^{sat} \subset  \Lambda =  \Lambda^{\vee} \subset ( \sL^{sat} )^{\vee} \subset  \sL_0^{\vee} \oplus \ZZ H^{\vee}.$$ 
 
 We have seen that $(\sK')^{\vee} \cong  \sL^{sat}/ \sL$, and that $\sK^{\vee} =  Coker (\Lambda \ra \sL_0^{\vee})$.

\begin{enumerate}
\item
For the Cayley cubic $$\sC_3 :=  \{ x | \sum_0^3 \frac{x_0 x_1 x_2 x_3}{x_j} = 0\}$$
the code $\sK = \ZZ/2 (\sum_1^4 e_i) \subset (\ZZ/2)^4$(well known, see also \cite{nodal}), and it equals $\sK'$ since $d=3$ is odd.
\item
 In the next examples, since the lattices $\sL, \Lambda$, have the same rank,  $  \sL^{sat}= \Lambda$, hence 
 $(\sK')^{\vee} = \Lambda/ \sL$.   Since $\Lambda$  is unimodular, and with the
 same rank as $\sL$,   $\Lambda$ corresponds to an isotropic
 subspace of the discriminant lattice 
 $$\Delta : = ( \sL' )^{\vee} /  \sL' = (( \sL'_0 )^{\vee} /  \sL' _0) \oplus ((\ZZ/3) H).$$ 
  \item
For the cubic with 3 cusps ($A_2$ singularities)  $$ \sY^{cu}_3 : = \{ x | x_0^3 -  x_1 x_2 x_3= 0\},$$
$\Delta \cong (\ZZ / 3)^3 \oplus \ZZ / 3 $, hence $\sK' \cong ( \ZZ/3)^2$, while $ \sK \cong \ZZ / 3$.
The last assertion because  $Y$ is a quotient $ Y = \PP^2 / (\ZZ / 3)$, with action having as fixed points exactly  the
3 coordinate points: hence $\pi_1(Y^*) = \ZZ/3$.

 $\Lambda/ \sL$ is generated by $ \frac{1}{3} ( E_1-E_2 - E_3 + E_4 + H)$
and by $ \frac{1}{3} ( E_1-E_2 - E_5 + E_6 + H)$.
\item
For the cubic with an $E_6$-singularity,  $$ \sY^{E_6}_3 : = \{ x | x_0 x_3^2 -   g(x_1, x_2, x_3)= 0\},$$
 $\Delta \cong \ZZ / 3 \oplus \ZZ / 3 $, and  $\sK' \cong \ZZ/3 $, $ \sK = 0 $.
 $\Lambda/ \sL$ is generated by   $ \frac{1}{3} ( E_1-E_2 - E_5 + E_6 + H)$.
\item
For the cubic with an $A_5$ singularity and a disjoint $A_1$ singularity,
$\Delta \cong ( \ZZ / 6  \oplus  \ZZ / 2)  \oplus  \ZZ / 3 $,
hence $\sK' \cong \ZZ/6 $ and  $\sK \cong \ZZ/2 $.
Here the generator of $\Lambda/ \sL$ can be directly  found using the representation
of $S$ as a blow up of the plane: $$\frac{1}{6} [ - 2 H + 3 E_6 - \sum_1^5 i E_i].$$
 
\end{enumerate}

Once we have the ancestors, we can find all the other codes by  geometric driven shortening,
in view of the fact that cubic surfaces with ADE singularities are unobstructed.

We shall next show a (very partial) list of ancestors for (non polarized) K3 surfaces,
based on the following Proposition.

\begin{prop}\label{torusquotient}
Let  $ Y= T / G$, where $G$ is a finite group acting on a  2-dimensional complex torus with a strictly positive
but finite number
of fixpoints, and such that $Y$ has only ADE singularities,
so that $Y$ has trivial canonical sheaf $\om_Y$.

Then $Y$ is an ancestor.

Moreover, let $ T = \CC^2 / \Lambda$,
and let 
$\Ga$ be the group of affine transformations such tht $ Y = \CC^2 / \Ga$,
so that we have an exact sequence $$ 1 \ra \Lambda \ra \Ga \ra G \ra1.$$

Then  the code group $\sK$ is isomorphic to $\Ga^{ab} $,
and we have an exact sequence $$ \Lambda_G \ra \Ga^{ab} \ra G^{ab} \ra1,$$
where $\Lambda_G$ is the group of $G$-covariants, the quotient of $\Lambda$ 
by the subgroup generated by $\{ \la - g \la| \la \in \Lambda \}$).

If $G$ is cyclic, then $\Lambda_G \hookrightarrow  \Ga^{ab}$. 

Moreover, if  
$\Ga = \Lambda  \rtimes G$, then 
 $\sK \cong G \oplus \Lambda_G$.
\end{prop}

\begin{proof}
If $Y$  were a partial smoothing  of another
 singular K3 surface $X$ with ADE singularities, then since $\pi_1 (X^*) \twoheadrightarrow \pi_1 (Y^*)$,
 then also $\pi_1 (X^*) $ would be infinite and by \cite{campana} also $X$ would be a Torus quotient,
and  the configuration of exceptional curves of $Y$ would be  a 
 geometric shortening of another torus quotient.
 
 Inspecting the list contained in Theorem 3.6 of \cite{gpp}, based on the results of \cite{fujiki},  we see that this cannot occur 
 for the cases (5)-(10), because the 
 number of exceptional curves is 19.
 
 For the  cases (1)-(4), one sees that the configuration of exceptional curves of $Y$ is not a
 geometric shortening of another one in the list.
 
 We have   $\pi_1(Y^*) = \Ga$, since $\CC^2$ is the orbifold universal covering of $Y^*$,
 hence $\sK$ is isomorphic to $\Ga^{ab} $.
 
 Now, let $\la_1, \dots, \la_4$ be a system of generators of $\Lambda$,  $ g_1, \dots , g_r$
 be a system of generators of $G$, and $ \ga_1, \dots , \ga_r$ be lifts of these generators
 to $\Ga$.
 
 Then $\Ga^{ab} $ has the above as generators, and we must add the relations 
 $$[\ga_j, \la_i]=1 = [\ga_h, \ga_j], \forall i=1,2,3,4, \ j,h, = 1, \dots, r.$$
 
 Then  the image of $\Lambda$ inside $\Ga^{ab}$ factors through $ \Lambda_G$; 
 and if $G$ is cyclic we have an exact sequence.
 
 In the case of a semidirect product, $G$ is a subgroup, hence we may take $\ga_j = g_j$. 
 
\end{proof}
\begin{ex}

For K3 surfaces, two ancestors are  well known (\cite{nikulin}, \cite{b-9}):

1)  the Kummer surfaces $ X= T / \pm 1$, with 16 $A_1$ singularities,
(16 nodes), which are quartic surfaces in $\PP^3$ if $T$ is a principally polarized 2-dimensional torus:
here $G = \{ \pm 1\}$, and  since $\Ga = \Lambda  \rtimes G$ we have $\sK \cong (\ZZ/2)^5$;

2) the K3 surfaces $Y= (F \times F)/ \ZZ/3$ with action generated by $(\zeta, \zeta^2)$. Here $F$ is the Fermat cubic,  $\zeta$ is
a primitive 3rd root of unity, and $Y$ has 9 cusps.
Here $\Lambda = (\ZZ \oplus \zeta \ZZ)^2$,  the covariants  $(\ZZ \oplus \zeta \ZZ)_G \cong \ZZ/3$,
again we have a semidirect product, hence $\sK \cong (\ZZ/3)^3$.

Barth \cite{b-8} proved  that the latter may only yield surfaces of degree $d \equiv 0,2 \ (mod \ 6)$.

3) Theorem 3.6 of \cite{gpp} contains a list of the exceptional configurations 
for 10 quotients of a 2-dimensional complex torus, the first two being
 the above  two examples 1) and 2). They are  all ancestors, by Proposition \ref{torusquotient}.
 
  In the cases (5)-(10) the number of exceptional curves is 19,
 in the cases (3) and (4) the  number of exceptional curves for  $Y$  is 18, but the configuration of exceptional curves of $Y$ is not a
 direct geometric shortening of another one in the list with 19 exceptional curves.

 4) Needless to say, an interesting question concerns  the degrees of the possible polarizations on the above  surfaces
 in the cases (3)-(10).
\end{ex}

\end{document}